\documentclass[12pt,a4paper]{article}
\usepackage[margin=30mm]{geometry}
\usepackage{amssymb,amsthm,amsmath, amsfonts,ascmac, enumerate,color}
\numberwithin{equation}{section}
\usepackage{graphicx}
\usepackage{adjustbox}
\usepackage[colorlinks=true,linkcolor=blue,citecolor=blue]{hyperref}

\theoremstyle{definition}
\newtheorem{thm}{Theorem}[section]
\newtheorem{defn}[thm]{Definition}
\newtheorem{lem}[thm]{Lemma}
\newtheorem{prop}[thm]{Proposition}
\newtheorem{rem}[thm]{Remark}

\newtheorem{ex}[thm]{Example}

\newcommand{\R}{\mathbb{R}}   
\newcommand{\C}{\mathbb{C}}   
\newcommand{\N}{\mathbb{N}}   
\newcommand{\Z}{\mathbb{Z}}   
\newcommand{\E}{\mathbb{E}}

\newcommand{\calA}{{\mathcal{A}}}

\newcommand{\calD}{{\mathcal{D}}}

\newcommand{\calH}{{\mathcal{H}}}

\newcommand{\calS}{{\mathcal{S}}}

\newcommand{\rmd}{\mathrm{d}}
\newcommand{\rmi}{\mathrm{i}}

\renewcommand{\Re}{{\sf Re}}
\renewcommand{\Im}{{\sf Im}}
\newcommand{\Ran}{{\sf Ran}}
\newcommand{\Ker}{{\sf Ker}}

\renewcommand{\Re}{\operatorname{Re}}
\renewcommand{\Im}{\operatorname{Im}}

\title{Poisson operator on the interacting Fock space associated with a discrete-time quantum walk}

\author{Daiju Funakawa, Yuki Ueda and Kazuyuki Wada}

\date{\today}

\begin{document}

\maketitle

\begin{abstract}
We study the Poisson operator on the interacting Fock space associated with a discrete-time quantum walk, which we call the QW-Poisson operator. First, we investigate the spectral properties of the QW-Poisson distribution. In particular, we establish a relation between the spectral distributions of the Poisson operator and the reversed Poisson operator on a general interacting Fock space via a size-biased transform. Next, we study the edge behavior of the density of the QW-Poisson distribution. We show that a phase transition occurs at the left endpoint of the support: depending on the parameter, the density either decays to $0$ or blows up to $+\infty$. Moreover, this phase transition coincides with the transition in the number of atoms of the QW-Poisson distribution, equivalently, in the point spectrum of the QW-Poisson operator, and with whether $0$ belongs to the spectrum of the QW-Poisson operator. Finally, we study a connection between the interacting Fock space associated with a discrete-time quantum walk and noncommutative probability theory. More precisely, we compute the moment-generating function and moments of the QW-Poisson operator, and obtain a limit theorem for the Konno distribution via a Poisson approximation. We also investigate the Boolean self-decomposability of the Konno distribution and the shifted reversed QW-Poisson distribution.
\end{abstract}

\tableofcontents

\section{Introduction}

\subsection{Quantum walks and related preliminaries}
Quantum walks are widely regarded as quantum analogs of classical random walks and play an important role in quantum search algorithms. Ambainis et al. \cite{ambainis2001one} studied discrete-time quantum walks from the perspective of quantum information theory. Beginning in the early 2000s, Konno \cite{konno2002quantum,konno2005new} investigated the asymptotic behavior of one-dimensional discrete-time quantum walks and established a weak limit theorem for their probability distributions.

We begin by introducing a one-dimensional discrete-time quantum walk and recalling Konno's weak limit theorem. The following introduction is based on \cite{suzukivelocity2016}.The Hilbert space we consider is $$\ell^{2}(\mathbb{\Z}; \C^2):=\left\{u: \Z\rightarrow \C^{2}\Big|\ \sum_{x\in\Z}\|u(x)\|_{\C^2}^{2}<\infty\right\},\quad u(x):=\begin{pmatrix}u^{(1)}(x) \\ u^{(2)}(x)\end{pmatrix},\ x\in\Z.$$
Next, we define the shift operator $S$ and the coin operator $C$ by
$$(Su)(x):=\begin{pmatrix}u^{(1)}(x+1) \\ u^{(2)}(x-1) \end{pmatrix},\ (Cu)(x):=\begin{pmatrix}a & b \\ c & d\end{pmatrix}\begin{pmatrix}u^{(1)}(x) \\ u^{(2)}(x)\end{pmatrix},\ u\in\ell^{2}(\Z; \C^{2}), \ x\in\Z,$$
where $a,b,c,d\in \C$ satisfy 
$$
|a|^2+|b|^2=|c|^2+|d|^2=|a|^2+|c|^2=|b|^2+|d|^2=1
$$
and 
$$
a\overline{b}+ c\overline{d}=a \overline{c}+b\overline{d}=0.
$$
These conditions imply that $C$ is unitary in $\ell^{2}(\Z; \C^{2}).$ The shift operator $S$ is also unitary in $\ell^{2}(\Z; \C^2).$ Under these settings, we introduce the time evolution operator of quantum walks by $U:=SC$. Since $S$ and $C$ are unitary, $U$ is also unitary. For any initial state $\Psi\in\ell^{2}(\Z; \C^2),$ the state after $t\in\Z$ is expressed by $U^{t}\Psi.$
We introduce two matrices $P$ and $Q$ as follows;
$$
P:=\begin{pmatrix}a & b \\ 0 & 0\end{pmatrix},\quad Q:=\begin{pmatrix}0 & 0 \\ c & d\end{pmatrix}.$$
Then, we can express the state after 1-step by
$$(U\Psi)(x)=P\Psi(x+1)+Q\Psi(x-1),\quad \Psi\in\ell^{2}(\Z; \C^2),\quad x\in \Z.$$
Intuitively, under the 1-step time evolution, a quantum walker moves to the left with weight $P$ and moves to the right with weight $Q$.
For nonnegative integers $l$ and $m$ satisfying $l+m=n$, let $\Xi_n(l,m)$ denote the sum of the matrix products corresponding to all trajectories consisting of $l$ steps to the left and $m$ steps to the right. More precisely,
$$
\Xi_n(l,m) = \sum_{\substack{l_1+\cdots + l_n=l\\ m_1+\cdots+m_n=m\\ l_j+m_j=1}} P^{l_1}Q^{m_1}P^{l_2}Q^{m_2}\cdots P^{l_n}Q^{m_n}.
$$
We consider the initial states localized in the origin with $\varphi=(\alpha,\beta)^T\in\C^2$ satisfying $|\alpha|^2+|\beta|^2=1$. We define a $\Z$-valued random variable $X_n$ by
$$
P(X_n=m-l)= \|\Xi_n(l,m)\varphi\|^2.
$$
The sequence $\{X_n\}_{n\ge 0}$ is called a \textit{one-dimensional discrete-time quantum walk with initial state $\varphi$}. In classical probability, the central limit theorem states that the centered position of a random walk, normalized by $\sqrt{n}$, converges in distribution to the normal distribution. Equivalently, the classical diffusive scaling is $1/\sqrt{n}$. In contrast, for discrete-time quantum walks, the appropriate scaling is ballistic, namely $1/n$. Konno established the following weak limit theorem for the rescaled position $X_n/n$, which may be viewed as a quantum-walk counterpart of the classical central limit theorem.

\begin{thm}[see \cite{konno2002quantum,konno2005new}]
Suppose that $abcd\neq 0$. Then, for any $u<v$,
$$
\lim_{n\to \infty} P\left(u\le \frac{X_n}{n} < v\right) =\int_u^v \{1-c(a,b,\varphi)x\} f_{\rm K}(x;|a|) \rmd x,
$$
where
$$
f_{\rm K}(x;r):= \frac{\sqrt{1-r^2}}{\pi(1+x^2)\sqrt{r^2-x^2}}\mathbf{1}_{(-r,r)}(x), \quad 0<r<1
$$
and
$c(a,b,\varphi)=|\alpha|^2-|\beta|^2+ \dfrac{a\alpha \overline{b\beta}+ \overline{a\alpha}b\beta}{|a|^2}$.
\end{thm}
In particular, if the parameters $a,b\in\mathbb{C}$ and the initial state $\varphi$ are chosen so that $c(a,b,\varphi)=0$, then the limiting distribution of $X_n/n$ is symmetric. We denote this probability distribution by
$$
{\rm K}_r(\rmd x):=f_{\rm K}(x;r)\rmd x, \qquad r\in (0,1),
$$
and call it the \textit{(symmetric) Konno distribution}. In this paper, we focus on the symmetric case. One reason for this choice is that, when viewed through its Jacobi parameters, the symmetric Konno distribution has a structure closely related to Wigner's semicircle law, one of the most important distributions in free probability theory, which was introduced by Voiculescu in the 1980s; see \cite{mingo2017free,nicaspeicher2006lecture, voiculescu1992free} and references therein for details. This connection makes the symmetric case particularly compatible with methods from noncommutative probability. To explain this point more precisely, we first recall the notion of Jacobi parameters.

Let $\rho$ be a probability measure on $\R$ with finite moments of all orders. Associated with $\rho$ are sequences $\{\alpha_n\}_n\subset\mathbb{R}$ and $\{\omega_n\}_n\subset[0,\infty)$ such that
$$
G_\rho(z) :=\int_{\R}\frac{1}{z-x}\rho(\rmd x)= \cfrac{1}{z-\alpha_1-\cfrac{\omega_1}{z-\alpha_2-\cfrac{\omega_2}{z-\alpha_3 - \cfrac{\omega_3}{z-\cdots}}}},
$$
where $G_{\rho}$ is called the {\it Cauchy transform} of $\rho$. We write
$$
J(\rho) = \begin{pmatrix}
\alpha_1& \alpha_2 & \alpha_3 &\cdots\\
\omega_1& \omega_2 & \omega_3 & \cdots
\end{pmatrix}
$$
and call $J(\rho)$ the \textit{Jacobi sequence} of $\rho$. Jacobi sequences play a fundamental role in the study of orthogonal polynomials and interacting Fock spaces associated with probability measures; see \cite{obata2017spectral} for further details. Hamada et al. \cite{hamada2009orthogonal} computed the Cauchy transform and the Jacobi sequence of the Konno distribution. Here, we restrict our attention to the symmetric case. The following result is given in \cite[Theorem 3.1]{hamada2009orthogonal}.
\begin{thm}
For any $r\in(0,1)$, the Cauchy transform of the symmetric Konno distribution is given by
\begin{align*}
G_{{\rm K}_r}(z) 
&=\frac{z(z^2-r^2)-\sqrt{1-r^2}\sqrt{z^2-r^2}}{(z^2-1)(z^2-r^2)}\\
&=\cfrac{1}{z-\cfrac{\omega_1}{z-\cfrac{\omega_2}{z-\cfrac{\omega}{z-\cfrac{\omega}{z-\cdots}}}}},
\end{align*}
where
\begin{align}\label{eq:omega}
\omega_1 = 1-\sqrt{1-r^2}, \quad \omega_2=\frac{\sqrt{1-r^2}}{2}\omega_1, \quad \omega=\frac{r^2}{4}.
\end{align}
Consequently, the Jacobi sequence of $\mathrm{K}_r$ is
$$
J({\rm K}_r)= \begin{pmatrix}
0& 0& 0& 0& \cdots\\
\omega_1& \omega_2& \omega& \omega& \cdots
\end{pmatrix}.
$$
\end{thm}
Thus, in the symmetric case, all diagonal Jacobi parameters vanish and the off-diagonal Jacobi parameters become constant from the third level onward. This eventual constancy is one of the reasons why the symmetric case is especially tractable within the framework of noncommutative probability. By contrast, the nonsymmetric case presents several technical difficulties for the methods developed in this paper; we discuss this issue in Section \ref{sec8}.

\subsection{Settings and main results}
Motivated by this Jacobi-parameter viewpoint, we now introduce interacting Fock spaces. This framework realizes a Jacobi sequence as the coefficients of creation and annihilation operators, and hence provides an operator-theoretic model for the corresponding probability distribution. In what follows, we focus on the interacting Fock space associated with the Jacobi sequence of the symmetric Konno distribution.
Fix $r\in(0,1)$, and let $\{\omega_n\}_n$ be the sequence defined in \eqref{eq:omega}. Let $\calH$ be a Hilbert space with a complete orthonormal system $\{\Phi_n\}_{n\geq 0}$. We define the following operators on $\calH$:
\begin{align}
        B_+^{QW}\Phi_n&:=\sqrt{\omega_{n+1}}\Phi_{n+1},\quad n\geq 0 \label{eq:creation}\\
        B_-^{QW}\Phi_n&:=\sqrt{\omega_n}\Phi_{n-1},\quad n\geq 1 \quad \text{and} \quad  B_-^{QW}\Phi_0:=0. \label{eq:anihilation}
\end{align}
The operators $B_+^{QW}$ and $B_-^{QW}$ are called the \textit{creation} and \textit{anihilation} operators with weight sequence $\{\omega_n\}_n$, respectively. Note that $(B_+^{QW})^\ast = B_-^{QW}$. The quadruple $\calH_{QW}:=(\calH, \{\Phi_n\}_{n\ge0}, B_+^{QW}, B_-^{QW})$ is called the \textit{interacting Fock space associated with a discrete-time quantum walk}. Two of the most important operators on $\calH_{QW}$ are the following:
\begin{itemize}
\item $N^{QW}:=B_+^{QW}+B_-^{QW}$ \quad (Gaussian operator)
\item $P_\lambda^{QW}:=(B_+^{QW}+\sqrt{\lambda}I)(B_-^{QW}+\sqrt{\lambda}I)$, \ $\lambda>0$ \quad (Poisson operator).
\end{itemize}
By the general theory of interacting Fock spaces, the Gaussian operator $N^{QW}$ has the Konno distribution $\mathrm{K}_r$ with respect to the vacuum state $\langle \Phi_0, \cdot \Phi_0\rangle_{\calH}$. See Section \ref{sec:IFQ} for further details.

In this paper, we study the spectral analysis of the Poisson operator $P_\lambda^{QW}$. We first investigate its spectral distribution $\mu_{P_\lambda^{QW}}$ with respect to the vacuum state $\langle \Phi_0,\cdot\Phi_0\rangle_{\calH}$. However, since there is a one-step discrepancy between the $\alpha$- and $\omega$-sequences in the Jacobi sequence of $\mu_{P_\lambda^{QW}}$ (see \eqref{eq:Jacobi-QWPOISSON} for details), it is difficult to compute the Cauchy transform of $\mu_{P_\lambda^{QW}}$ directly and to obtain its distributional properties. To overcome this difficulty, we first focus on the reversed QW-Poisson operator:
$$
C_\lambda^{QW}:=(B_-^{QW}+\sqrt{\lambda}I)(B_+^{QW}+\sqrt{\lambda}I).
$$
The Jacobi sequence of $\mu_{C_\lambda^{QW}}$ is simpler than those of $\mu_{P_\lambda^{QW}}$, and hence it is possible to obtain the distributional properties of $\mu_{C_\lambda^{QW}}$; see Theorem \ref{thm:C}.

Furthermore, we establish a relation between the spectral distributions of the Poisson operator and the reversed Poisson operator via the size-biased transform; see Theorem \ref{thm:Poisson_size-biased}. This theorem is not restricted to the QW case. Rather, it is formulated for Poisson operators associated with more general Jacobi sequences and gives a general relation between the spectral distribution of a Poisson operator and that of the corresponding reversed Poisson operator. Therefore, Theorem \ref{thm:Poisson_size-biased} may be of independent interest. Finally, combining these results, we obtain the following theorem; see Lemma \ref{lem:basicproperties_f_g} (1), Proposition \ref{prop:equation_zero}, Remark \ref{rem:S_-=x_0} and Theorem \ref{thm:spectral_P} for details.

\begin{thm}\label{thm:mainFUW}
Let $0<r<1$ and $\lambda>0$. Define $C_\lambda^{QW}:=(B_-^{QW}+\sqrt{\lambda}I)(B_+^{QW}+\sqrt{\lambda}I)$. Then the absolutely continuous part of the spectral distribution of $C_\lambda^{QW}$ with respect to Lebesgue measure, denoted by $p_{C_\lambda^{QW}}(x)\rmd x$, is supported on $[S_-,S_+]:=[\lambda + \omega - 2\sqrt{\lambda\omega}, \lambda+\omega+2\sqrt{\lambda\omega}]$. 
Moreover, let $\Lambda_-(r)<\Lambda_+(r)$ be the quantities defined in \eqref{eq:Lambda}. For $\lambda\le \omega$, define
$$
R_{\lambda,\omega}:=\max\left\{\frac{\omega_1\omega_2(1-\lambda/\omega)}{\lambda^2+\omega_2(\lambda+\omega_1)(1-\lambda/\omega)},0\right\},
$$
where $R_{\lambda,\omega}=0$ if and only if $\lambda=\omega$.
\begin{enumerate}[\rm (1)]
\item If $0<\lambda \le \Lambda_-(r)$, then there exist unique points $x_0 \in \sigma(C_\lambda^{QW}) \cap (0,S_-]$ and $x_1 \in \sigma(C_\lambda^{QW}) \cap (\lambda +\omega_1, (\sqrt{\lambda}+\sqrt{\omega_1})^2)$ such that
\begin{align*}
\mu_{P_\lambda^{QW}}(\rmd x) &= R_{\lambda,\omega}\delta_0(\rmd x)+\lambda x_0^{-1} \mu_{C_\lambda^{QW}}(\{x_0\})\delta_{x_0}(\rmd x) \\
&\hspace{4mm}+ \lambda x_1^{-1} \mu_{C_\lambda^{QW}}(\{x_1\})\delta_{x_1}(\rmd x) + \lambda x^{-1} p_{C_\lambda^{QW}}(x)\rmd x,
\end{align*}
where we note that $x_0=S_->0$ when $\lambda =\Lambda_-(r)$.
\end{enumerate}
Assume that $\omega <\Lambda_+(r)$, or equivalently, $\dfrac{2\sqrt{2}}{3}<r<1$. Then
\begin{enumerate}
\item[\rm (2)] If $\Lambda_-(r) < \lambda \le \omega$, then there exists a unique point $x_1 \in \sigma(C_\lambda^{QW}) \cap (\lambda +\omega_1, (\sqrt{\lambda}+\sqrt{\omega_1})^2)$ such that
$$
\mu_{P_\lambda^{QW}}(\rmd x) =  R_{\lambda,\omega}\delta_0(\rmd x)+\lambda x_1^{-1} \mu_{C_\lambda^{QW}}(\{x_1\})\delta_{x_1}(\rmd x) + \lambda x^{-1} p_{C_\lambda^{QW}}(x)\rmd x.
$$

\item[\rm (3)]  If $\omega < \lambda < \Lambda_+(r)$, then there exists a unique point $x_1 \in \sigma(C_\lambda^{QW}) \cap (\lambda +\omega_1, (\sqrt{\lambda}+\sqrt{\omega_1})^2)$ such that
$$
\mu_{P_\lambda^{QW}}(\rmd x) =  \lambda x_1^{-1} \mu_{C_\lambda^{QW}}(\{x_1\})\delta_{x_1}(\rmd x) + \lambda x^{-1} p_{C_\lambda^{QW}}(x)\rmd x.
$$
\item[\rm (4)] If $\lambda \ge \Lambda_+(r)$, then there exists a unique point $x_2 \in \sigma(C_\lambda^{QW}) \cap (S_+, (\sqrt{\lambda}+\sqrt{\omega_1})^2)$ such that
$$
\mu_{P_\lambda^{QW}}(\rmd x) =\lambda x_2^{-1} \mu_{C_\lambda^{QW}}(\{x_2\})\delta_{x_2}(\rmd x) + \lambda x^{-1} p_{C_\lambda^{QW}}(x)\rmd x.
$$
\end{enumerate}
Assume that $\Lambda_+(r) \le \omega$, or equivalently, $0<r\le \dfrac{2\sqrt{2}}{3}$. Then
\begin{enumerate}
\item[\rm (2)'] If $\Lambda_-(r) < \lambda <\Lambda_+(r)$, then there exists a unique point $x_1 \in \sigma(C_\lambda^{QW}) \cap (\lambda +\omega_1, (\sqrt{\lambda}+\sqrt{\omega_1})^2)$ such that
$$
\mu_{P_\lambda^{QW}}(\rmd x) =  R_{\lambda,\omega}\delta_0(\rmd x)+\lambda x_1^{-1} \mu_{C_\lambda^{QW}}(\{x_1\})\delta_{x_1}(\rmd x) + \lambda x^{-1} p_{C_\lambda^{QW}}(x)\rmd x.
$$
\item[\rm (3)'] If $\Lambda_+(r) \le \lambda \le \omega$, then there exists a unique point $x_2 \in \sigma(C_\lambda^{QW}) \cap (S_+, (\sqrt{\lambda}+\sqrt{\omega_1})^2)$ such that
\begin{align*}
\mu_{P_\lambda^{QW}}(\rmd x) 
= R_{\lambda,\omega} \delta_0(\rmd x)&+ \omega x_2^{-1} \mu_{C_\lambda^{QW}}(\{x_2\})\delta_{x_2}(\rmd x) \\
&+ \lambda x^{-1} p_{C_\lambda^{QW}}(x)\mathbf{1}_{(0,\infty)}(x)\rmd x.
\end{align*}

\item[\rm (4)'] If $\lambda >\omega$, then there exists a unique point $x_2 \in \sigma(C_\lambda^{QW}) \cap (S_+, (\sqrt{\lambda}+\sqrt{\omega_1})^2)$ such that
\begin{align*}
\mu_{P_\lambda^{QW}}(\rmd x) =  \lambda x_2^{-1} \mu_{C_\lambda^{QW}}(\{x_2\})\delta_{x_2}(\rmd x) + \lambda x^{-1} p_{C_\lambda^{QW}}(x)\mathbf{1}_{(0,\infty)}(x)\rmd x.
\end{align*}
\end{enumerate}
The value $\lambda=\Lambda_-(r)$ marks a phase transition between the regime in which the QW-Poisson distribution has three atoms and the regime in which it does not. Likewise, the value $\lambda=\omega$ marks a phase transition between the regime in which the QW-Poisson distribution has an atom at $0$ and the regime in which it does not. Therefore, these parameter values may be regarded as \textit{spectral phase-transition points}.
\end{thm}

Next, we focus on the absolutely continuous part of the QW-Poisson operator. More precisely, we investigate the edge behavior of its density function; see Theorems \ref{thm:edge_behabior_density} and \ref{thm:edge_behavior}.

\begin{thm}\label{thm:mainFUW2}
 Fix $0<r<1$. The following properties hold:
\begin{enumerate}
\item[\rm (1)] For any $\lambda>0$, we get
$\lim_{x\uparrow S_+} \left(\dfrac{\rmd \mu_{P_\lambda^{QW}}}{\rmd x}\right)'(x) = -\infty$.
\item[\rm (2)] If $\lambda = \Lambda_-(r)$ or $\lambda = \omega$, then we get $\lim_{x\downarrow S_-} \left(\dfrac{\rmd \mu_{P_\lambda^{QW}}}{\rmd x}\right)'(x) = -\infty$.
\item[\rm (3)] If $\lambda \neq \Lambda_-(r)$ and $\lambda \neq \omega$, then we get $\lim_{x\downarrow S_-} \left(\dfrac{\rmd \mu_{P_\lambda^{QW}}}{\rmd x}\right)'(x) = \infty$. 
\end{enumerate}
The spectral phase-transition points $\lambda = \Lambda_-(r)$ and $\lambda=\omega$ are also precisely the phase-transition points for the edge behavior of the density of $\mu_{P_\lambda^{QW}}$ at the left endpoint $S_-$ of its support. In more detail, we obtain
$$
\frac{\rmd \mu_{P_\lambda^{QW}}}{\rmd x}(x) \sim C_\lambda (x-S_-)^{\kappa(\lambda)} \qquad x\downarrow S_-,
$$
where $C_\lambda>0$ and
$$
\kappa(\lambda):=\begin{cases}
\dfrac{1}{2}, & \lambda\neq \Lambda_-(r) \quad \text{and} \quad \lambda\neq \omega,\\\\
-\dfrac{1}{2}, & \lambda = \Lambda_-(r) \quad \text{or} \quad \lambda = \omega.
\end{cases}
$$
\end{thm}

\subsection*{Previous works and new perspectives on QW and NCPT}
Previous work has established connections between discrete-time quantum walks and noncommutative probability. For example, the localization of the Grover walk on spidernets was analyzed through interacting Fock spaces and free Meixner laws in \cite{konno2013localization}. Our approach is different: we construct an interacting Fock space directly from the Jacobi parameters of the Konno distribution for a one-dimensional discrete-time quantum walk and study the associated Poisson and reversed Poisson operators. To the best of our knowledge, the present work provides the first systematic spectral analysis of these Poisson-type operators in a quantum walk derived interacting Fock space.
The key point is that the interacting Fock space associated with the Konno distribution allows us to treat quantum-walk limit distributions and the corresponding Poisson operators within a unified operator-theoretic framework. We expect that this perspective will be useful for further developing connections between quantum walks and noncommutative probability theory.

\subsection*{Organization of the paper}
In Section \ref{sec2}, we introduce the Cauchy transform, interacting Fock spaces, and related preliminary results. In Section \ref{sec3}, we study the spectral analysis of the reversed QW-Poisson operator $C_\lambda^{QW}$. In Section \ref{sec4}, we establish a formula relating the spectral distributions of the Poisson operator and the reversed Poisson operator via a size-biased transform, and apply it to $\mu_{C_\lambda^{QW}}$ to obtain the spectral properties of $\mu_{P_\lambda^{QW}}$. In Section \ref{sec5}, we investigate the edge behavior of density of $\mu_{P_\lambda^{QW}}$. In Section \ref{sec6}, we investigate the QW-Poisson operator from the perspective of noncommutative probability theory. More precisely, in Section \ref{sec6-1}, we compute the moment-generating function and the first four moments of the QW-Poisson distribution (Theorem \ref{thm:moments}). In Section \ref{sec6-2}, we establish a limit theorem for the Konno distribution through a Poisson approximation (Theorem \ref{thm:Konno-Poisson}). In Section \ref{sec6-3}, we investigate the Boolean self-decomposability of the Konno distribution (Theorem \ref{thm:BSD__Konno}) and the shifted reversed QW-Poisson distribution (Theorem \ref{thm:BSD_CPoisson}).

%%%SECTION2%%%
\section{Preliminaries}
\label{sec2}

\subsection{Cauchy transform}

For a probability measure $\mu$ on $\R$, we define the {\it Cauchy transform} $G_\mu$ of $\mu$ by
$$
    G_{\mu}(z) = \int_{\mathbb{R}} \frac{1}{z-t}\,\mu(\rmd t), 
    \qquad z \in \C^+.
$$
It is well known that $G_\mu$ is analytic on $\mathbb{C}^+$ and maps $\mathbb{C}^+$ into $\mathbb{C}^-$. The Cauchy transform plays a fundamental role in determining the distributional properties of $\mu$. In particular, the following results are useful for characterizing the measure $\mu$.

\begin{prop}[Stieltjes-inversion formula, see \cite{schmudgen2012unbounded} e.g.]\label{prop:Stieltjes}
Let $\mu$ be a finite Borel measure on $\R$ and $-\infty < a<b<\infty$. Then we have
\begin{enumerate}[\rm (1)]
\item $\displaystyle \mu((a,b))+ \frac{1}{2}(\mu(\{a\})+\mu(\{b\})) = -\frac{1}{\pi} \lim_{y\downarrow 0} \int_a^b \Im[G_\mu(x+\rmi y)] dx$.
\item $\mu(\{a\})=\lim_{y\downarrow 0}\rmi y G_\mu(a+\rmi y)$.
\end{enumerate}
In particular, if $G_\mu$ extends continuously from $\C^+$ to $\C^+ \cup [a,b]$, then $\mu|_{[a,b]}$ is absolutely continuous with respect to Lebesgue measure, and its density $p_\mu$ is given by
$$
p_\mu(x) = -\frac{1}{\pi} \Im[G_\mu(x)], \qquad x\in [a,b].
$$
\end{prop}

We next present some important examples of probability measures that will be used throughout this paper.
\begin{ex}\label{ex:distributions}
\begin{enumerate}[\rm (1)]
\item We define the {\it semicircle law} with mean $m\in \R$ and variance $v>0$ by
$$
{\rm SC}_{m,v}( \rmd x) := \frac{1}{2\pi v}\sqrt{4v- (x-m)^2} \cdot \mathbf{1}_{[m-2\sqrt{v},m+2\sqrt{v}]}(x) \rmd x.
$$
The semicircle law is known to arise as the limiting distribution in the free analogue of the central limit theorem. According to \cite{nicaspeicher2006lecture}, its Cauchy transform is given by
\begin{align*}
    G_{{\rm SC}_{m,v}}(z) = \frac{z-m-\sqrt{(z-m)^2 - 4v}}{2}, \qquad z \in \C^+.
\end{align*}
Here, the branch of the square root is chosen so that $\sqrt{z^2-4v}\sim z$ as $z$ tends to infinity non-tangentially. Moreover, the Jacobi sequence of $\mu_{{\rm SC}_{m,v}}$ is given by
$$
J({\rm SC}_{m,v}) = \begin{pmatrix}
m& m & m& \cdots\\
v& v & v & \cdots
\end{pmatrix}.
$$
Equivalently, we have
$$
G_{{\rm SC}_{m,v}}(z) = \cfrac{1}{z-m- \cfrac{v}{z-m-\cfrac{v}{z-\cdots}}}.
$$
\item The \textit{free Meixner distribution}, denoted by ${\rm fM}_{s,a,b}$, is the probability measure characterized by the following Jacobi sequence:
$$
J({\rm fM}_{s,a,b})= \begin{pmatrix}
0 & a & a & a &\cdots\\
s & s+b & s+b & s+b& \cdots
\end{pmatrix},
$$
where $s\ge 0$, $a\in \R$ and $b\ge -1$.
Its Cauchy transform is given by
\begin{align*}
G_{{\rm fM}_{s,a,b}}(z) 
&= \cfrac{1}{z- \cfrac{s}{z-a-\cfrac{s+b}{z-a-\cfrac{s+b}{z-a-\cfrac{s+b}{z-\cdots}}}}}\\
&=\frac{(s+2b)z+sa-s\sqrt{(z-a)^2-4(s+b)}}{2(bz^2+saz+s^2)}.
\end{align*}
The class of free Meixner distributions includes many important probability distributions; see \cite{anshelevich2003free, bozejko2006class, saitoh2001infinite} for further details.
\end{enumerate}
\end{ex}

\subsection{Interacting Fock space associated with a quantum walk}\label{sec:IFQ}
For $0<r<1$, consider the sequence $\{\omega_n\}_n$ defined in \eqref{eq:omega}. In this section, we study the interacting Fock space $\calH_{QW}=(\calH, \{\Phi_n\}_{n\ge0}, B_+^{QW},B_-^{QW})$, where $B_+^{QW}, B_-^{QW}$ are defined in \eqref{eq:creation} and \eqref{eq:anihilation}, respectively.

It is easy to verify that the following inequality holds.
\begin{lem}\label{lem:omega}
For $0<r<1$, we have $\omega_2<\omega<\omega_1$.
\end{lem}
\begin{proof}
A direct computation shows
    \begin{align*}
        \omega-\omega_2=\frac{r^2}{4}-\frac{\sqrt{1-r^2}-1+r^2}{2}=\frac{-2\sqrt{1-r^2}+2-r^2}{4}=\frac{\omega_1^2}{4}>0.
    \end{align*}
Moreover,
    \begin{align*}
        \omega_1-\omega=1-\sqrt{1-r^2}-\frac{r^2}{4}=\frac{(1-\sqrt{1-r^2})(3-\sqrt{1-r^2})}{4}>0.
    \end{align*}
\end{proof}

\begin{lem}\label{lem:bounded}
The operators $B_+^{QW}, B_-^{QW}$ on $\calH$ are bounded.
\end{lem}
\begin{proof}
For any $\xi = \sum_{n\ge 0} c_n \Phi_n \in \calH$ ($c_n\in \C$, $n\ge0$), we have
$$
B_+^{QW} \xi = \sum_{n\ge 0} c_n \sqrt{w_{n+1}}\Phi_{n+1}.
$$
Since $\sup_{n\ge 1} \omega_n = \omega_1$ by Lemma \ref{lem:omega}, we get
$$
\|B_+^{QW} \xi\|^2= \sum_{n \ge 0} |c_n|^2 w_{n+1} \le w_1 \sum_{n\ge0}|c_n|^2 = w_1 \|\xi\|^2.
$$
Therefore $\|B_+^{QW}\|\le \sqrt{w_1}$. Since $\|B_-^{QW}\|=\|(B_+^{QW})^\ast\|=\|B_+^{QW}\|$, the operator $B_-^{QW}$ is also bounded.
\end{proof}

Let $\calA_{QW}$ be a unital $C^\ast$-algebra generated by $B_+^{QW}, B_-^{QW}$, and let $\varphi_{QW}:\calA_{QW} \to \C$ be the vacuum state defined by
$$
\varphi_{QW}(a):=\langle \Phi_0, a\Phi_0\rangle_{\calH}, \qquad a\in \calA_{QW}.
$$
The pair $(\calA_{QW},\varphi_{QW})$ is a $C^\ast$-probability space. By \cite[Proposition 3.13]{nicaspeicher2006lecture}, for a normal operator $a\in \calA_{QW}$, there exists a unique compactly supported probability measure $\mu_a\equiv \mu_a^{QW}$ on $\C$ such that
$$
\varphi_{QW}(a^k) = \int_{\C} z^k \mu_a(\rmd z), \qquad k\ge0.
$$
We call the probability measure $\mu_a$ {\it (spectral) distribution of $a$}. Note that $\text{supp}(\mu_a) \subset \sigma(a) \subset \{z\in \C: |z|\le \|a\|\}$.

We introduce one of the most important operators on the interacting Fock space, namely, the \textit{QW-gaussian operator}:
$$
N^{QW}:= B_+^{QW} + B_-^{QW}.
$$
Clearly, the operator $N^{QW}$ is self-adjoint, and hence its spectral measure $\mu_{N^{QW}}$ is supported on $\R$. We call the measure $\mu_{N^{QW}}$ the \textit{QW-gaussian distribution}. By \cite[Theorem 4.13]{obata2017spectral}, we immediately obtain the following result.
\begin{prop}
We have $\mu_{N^{QW}}={\rm K}_r$.
\end{prop}
Thus, the Konno distribution ${\rm K}_r$ plays a crucial role in the interacting Fock space $(\calH, \{\Phi_n\}_{n\ge0}, B_+^{QW}, B_-^{QW})$.

%SECTION 3

\section{Reversed QW-Poisson operator}
\label{sec3}

For $\lambda>0$, we define the following operator on the interacting Fock space $\calH_{QW}$:
$$
C_\lambda^{QW}:=(B_-^{QW}+\sqrt{\lambda}I)(B_+^{QW}+\sqrt{\lambda}I).
$$
This operator is the partner operator of the Poisson operator. We call $C_\lambda^{QW}$ the \textit{reversed QW-Poisson operator}.
Clearly, $C_\lambda^{QW}$ is a positive operator in $\calA_{QW}$. Hence, there exists a unique spectral distribution $\mu_{C_\lambda^{QW}}$ associated with $C_\lambda^{QW}$. We call the measure $\mu_{C_\lambda^{QW}}$ the \textit{reversed QW-Poisson distribution}. By Lemma \ref{lem:bounded}, we have
$$
\|C_\lambda^{QW}\| \le \|B_-^{QW}+\sqrt{\lambda}I\|\|B_+^{QW}+\sqrt{\lambda}I\| \le (\sqrt{\omega_1}+\sqrt{\lambda})^2.
$$
Therefore, $\sigma(C_\lambda^{QW}) \subset [0,(\sqrt{\omega_1}+\sqrt{\lambda})^2]$. Following Theorem \ref{thm:C} and Proposition \ref{prop:final_spec}, we further investigate the spectrum of $C_\lambda^{QW}$.

First, we determine the Jacobi sequence of $\mu_{C_\lambda^{QW}}$.
\begin{lem}\label{lem:Jacobi_C}
The Jacobi sequence of $\mu_{C_\lambda^{QW}}$ is given by
\begin{align}\label{eq:Jacobi_Poisson}
J(\mu_{C_\lambda^{QW}}) = \begin{pmatrix}
\widetilde{\alpha}_1& \widetilde{\alpha}_2& \widetilde{\alpha}& \widetilde{\alpha}& \cdots\\
\widetilde{\omega}_1& \widetilde{\omega}_2& \widetilde{\omega}& \widetilde{\omega}& \cdots
\end{pmatrix},
\end{align}
where
\begin{align*}
&\widetilde{\alpha}_1=\lambda+\omega_1,\quad \widetilde{\alpha}_2=\lambda+\omega_2,\quad \widetilde{\alpha}=\lambda+\omega,\\
&\widetilde{\omega}_1=\lambda\omega_1,\quad \widetilde{\omega}_2=\lambda\omega_2,\quad \widetilde{\omega}=\lambda\omega.\\
\end{align*}
\end{lem}
\begin{proof}
For $n\ge 0$, it is easy to see that
\begin{align}\label{eq:C_zenkashiki}
C_\lambda^{QW} \Phi_n = \sqrt{\lambda \omega_{n+1}}\Phi_{n+1}+ (\lambda + \omega_{n+1}) \Phi_n + \sqrt{\lambda \omega_n}\Phi_{n-1},
\end{align}
where $\omega_0=0$ and $\omega_n =\omega$ for all $n\ge 3$.
Therefore, the operator $C_\lambda^{QW}$ admits the following matrix representation as a Jacobi matrix:
$$
C_\lambda^{QW} =
\begin{pmatrix}
\lambda + \omega_1& \sqrt{\lambda\omega_1}& 0 & 0& 0 & \cdots\\
\sqrt{\lambda \omega_1}& \lambda + \omega_2 & \sqrt{\lambda \omega_2} & 0 & 0& \cdots\\
0 & \sqrt{\lambda \omega_2} & \lambda+\omega & \sqrt{\lambda \omega}& 0 & \cdots\\
0 & 0 & \sqrt{\lambda \omega} & \lambda + \omega & \sqrt{\lambda \omega} & \cdots\\
\vdots& \vdots&\vdots & \ddots & \ddots & \ddots
\end{pmatrix}
$$
with respect to the basis $\{\Phi_n\}_{n\ge0}$. Therefore, we obtain
\begin{align*}
G_{\mu_{C_\lambda^{QW}}}(z) 
&= \varphi_{QW}((z-C_\lambda^{QW})^{-1}) \\
&= \cfrac{1}{z- (\lambda+\omega_1) - \cfrac{\sqrt{\lambda\omega_1}}{z-(\lambda +\omega_2)- \cfrac{\sqrt{\lambda \omega_2}}{z-(\lambda+\omega)- \cfrac{\sqrt{\lambda\omega}}{z-(\lambda+\omega)-\cfrac{\sqrt{\lambda \omega}}{z-\cdots}}}}}.
\end{align*}
Consequently, the Jacobi sequence of $\mu_{C_\lambda^{QW}}$ is given by \eqref{eq:Jacobi_Poisson}.
\end{proof}

The Cauchy transform of $\mu_{C_\lambda^{QW}}$ can be expressed in terms of the Cauchy transform of either the semicircle law or a free Meixner distribution. More precisely, for $\lambda>0$, we have
\begin{align}
G_{\mu_{C_\lambda^{QW}}}(z) 
&= \cfrac{1}{z-\widetilde{\alpha}_1-\cfrac{\widetilde{\omega}_1}{z-\widetilde{\alpha}_2- \widetilde{\omega_2}G_{{\rm SC}_{\widetilde{\alpha},\widetilde{\omega}}}(z)}}  \label{eq:Cauchy_Poisson1}
\\
&= \cfrac{1}{z-\widetilde{\alpha}_1- \widetilde{\omega}_1G_{{\rm fM}_{\widetilde{\omega}_2,\widetilde{\alpha}-\widetilde{\alpha}_2, \widetilde{\omega}-\widetilde{\omega}_2}}(z-\widetilde{\alpha}_2)}. \label{eq:Cauchy_Poisson}
\end{align}

From the above expression for the Cauchy transform, we obtain a detailed description of the reversed QW-Poisson distribution $\mu_{C_\lambda^{QW}}$. To state the result, we introduce the following notation:
\begin{align*}
    t:=\sqrt{1-r^2} \qquad  \text{for} \quad 0<r<1
\end{align*}
and
  \begin{align}\label{eq:Lambda}
    \Lambda_-(r):=\left(\frac{r(1-t)(3-t)}{4(t^2-t+2)}\right)^2,\quad 
       \Lambda_{+}(r):=\left(\frac{(1-t)(3-t)}{4r}\right)^2.
    \end{align}
Moreover, We define the following functions:
\begin{align*}
    f(x):=\frac{1}{x-\widetilde{\alpha}_1}-\frac{1}{\widetilde{\omega}_1}\Bigl(x-\widetilde{\alpha}_2-\widetilde{\omega}_2\frac{x-\widetilde{\alpha}+\sqrt{(x-\widetilde{\alpha})^2-4\widetilde{\omega}}}{2\widetilde{\omega}}\Bigr),\quad x<\widetilde{\alpha}-2\sqrt{\widetilde{\omega}};\\
    g(x):=\frac{1}{x-\widetilde{\alpha}_1}-\frac{1}{\widetilde{\omega}_1}\Bigl(x-\widetilde{\alpha}_2-\widetilde{\omega}_2\frac{x-\widetilde{\alpha}-\sqrt{(x-\widetilde{\alpha})^2-4 \widetilde{\omega}}}{2\widetilde{\omega}}\Bigr),\quad x>\widetilde{\alpha}+2\sqrt{\widetilde{\omega}}.
\end{align*}

\begin{lem}\label{lem:basicproperties_f_g}
The following properties hold.
\begin{enumerate}[\rm (1)]
\item For any $r\in (0,1)$, we have $\Lambda_-(r) < \Lambda_+(r)$ and $\Lambda_-(r) < \omega$. Furthermore, $\omega < \Lambda_+(r)$ if and only if $t<\dfrac{1}{3}$, i.e., $\dfrac{2\sqrt{2}}{3}<r<1$.
\item $\widetilde{\alpha}-2\sqrt{\widetilde{\omega}}<\widetilde{\alpha}_1$. Moreover, $\widetilde{\alpha}-2\sqrt{\widetilde{\omega}}>0$ if and only if $\lambda \neq \omega$. 
\item $\widetilde{\alpha}_1 < \widetilde{\alpha}+2\sqrt{\widetilde{\omega}}$ if and only if $\lambda > \Lambda_+(r)$.
\item If $\lambda \neq \omega$, then $f$ is continuous and monotonically decreasing on on $(0, \widetilde{\alpha} - 2\sqrt{\widetilde{\omega}})$.
\item $g$ is continuous on $(\widetilde{\alpha} + 2\sqrt{\widetilde{\omega}}, \infty)$ if and only if $\lambda>\Lambda_+(r)$. In this case, $g$ is monotonically decreasing on $(\widetilde{\alpha} + 2\sqrt{\widetilde{\omega}}, \infty)$.
\item If $\lambda\le\Lambda_+(r)$, then $g$ is monotonically decreasing on each region, $(\widetilde{\alpha}+2\sqrt{\omega}, \widetilde{\alpha}_1)$ and $(\widetilde{\alpha}_1, \infty)$. 
\item If $\lambda \neq \omega$, then $f$ admits a continuous extension to $[0,\widetilde{\alpha}-2\sqrt{\widetilde{\omega}}]$, and $f(0)>0$.
\item $g((\sqrt{\omega_1}+\sqrt{\lambda})^2)<0$.
\end{enumerate}
\end{lem}
\begin{proof}
A straightforward computation establishes properties (1)--(3).
\begin{enumerate}
\item[\rm (4)] By (2), it is easy to see that $f$ is differentiable on $(0,\widetilde{\alpha}-2\sqrt{\widetilde{\omega}})$. Moreover, we have
    \begin{align*}
    f'(x)&=-\left(\frac{1}{x-\widetilde{\alpha}_1}\right)^2-\frac{1}{\widetilde{\omega}_1}\left(1-\widetilde{\omega}_2\frac{1+\frac{x-\widetilde{\alpha}}{\sqrt{(x-\widetilde{\alpha})^2-4\widetilde{\omega}}}}{2\widetilde{\omega}}\right)    \\
    &=-\left(\frac{1}{x-\tilde{\alpha}_1}\right)^2-\frac{1}{\tilde{\omega}_1}\left(1-\widetilde{\omega}_2\frac{\sqrt{(x-\widetilde{\alpha})^2-4\widetilde{\omega}}+x-\widetilde{\alpha}}{2\widetilde{\omega}\sqrt{(x-\widetilde{\alpha})^2-4\widetilde{\omega}}}\right) 
    \end{align*}
Using $x<\widetilde{\alpha}-2\sqrt{\widetilde{\omega}}$, we get $f'(x)<0$.

\item[\rm (5)] By (3), $g$ is continuous on $(\widetilde{\alpha}+2\sqrt{\widetilde{\omega}},\infty)$ if and only if $\lambda>\Lambda_+(r)$. Since
\begin{align*}
    g'(x)&=-\left(\frac{1}{x-\widetilde{\alpha}_1} \right)^2-\frac{1}{\widetilde{\omega}_1}\left(1-\widetilde{\omega}_2\frac{\sqrt{(x-\widetilde{\alpha})^2-4\widetilde{\omega}}-(x-\widetilde{\alpha})}{2\widetilde{\omega}\sqrt{(x-\widetilde{\alpha})^2-4\widetilde{\omega}}}\right),
    \end{align*}
and $x>\widetilde{\alpha}+2\sqrt{\widetilde{\omega}}$, we have $g'(x)<0$ on $(\widetilde{\alpha}+2\sqrt{\widetilde{\omega}}, \infty)$.
\\
\item[\rm (6)] If $\lambda\le \Lambda_+(r)$, then $g$ is discontinuous at $\widetilde{\alpha}_1$. However, since $g'(x)<0$ for $x>\widetilde{\alpha}+2\sqrt{\widetilde{\omega}}$, it follows that $g$ is monotonically decreasing on $(\widetilde{\alpha}+2\sqrt{\omega}, \widetilde{\alpha}_1)\cup(\widetilde{\alpha}_1, \infty)$.

\item[\rm (7)] It is easy to see that $f$ admits a continuous extension to  $[0,\widetilde{\alpha}-2\sqrt{\widetilde{\omega}}]$. Moreover, we have
\begin{align*}
f(0)
&= -\frac{1}{\widetilde{\alpha}_1}- \frac{1}{\widetilde{\omega}_1}\left( -\widetilde{\alpha}_2-\widetilde{\omega}_2 \cdot \frac{-\widetilde{\alpha}+\sqrt{\widetilde{\alpha}^2-4\widetilde{\omega}}}{2\widetilde{\omega}}\right)\\
&=-\frac{1}{\lambda+\omega_1}-\frac{1}{\lambda \omega_1} \underbrace{\left(-(\lambda+\omega_2) - \frac{\omega_2}{2\omega}\left(-(\lambda+\omega)+|\lambda-\omega| \right) \right)}_{=:I(\lambda)}.
\end{align*}
If $\lambda>\omega$, then it is easy to see that $I(\lambda)=-\lambda$, and hence
\begin{align*}
f(0)= -\frac{1}{\lambda+ \omega_1} + \frac{1}{\omega_1} >0.
\end{align*}
If $\lambda <\omega$, then
\begin{align*}
I(\lambda) = -(\lambda+\omega_2)+\frac{\omega_2}{\omega}\lambda.
\end{align*}
Theorefore, we have
\begin{align*}
f(0)
&= -\frac{1}{\lambda+\omega_1}-\frac{1}{\lambda \omega_1} \left(-(\lambda+\omega_2)+\frac{\omega_2}{\omega}\lambda \right)\\
&=-\dfrac{1}{\lambda+w_1}+\dfrac{1}{w_1}+\dfrac{w_2}{\lambda w_1}-\dfrac{w_2}{w_1 w}
\\
&=-\dfrac{1}{\lambda+w_1}+\dfrac{1}{w_1}+\dfrac{w_2}{w_1}\left(\dfrac{1}{\lambda}-\dfrac{1}{w}\right)>0.
\end{align*}
\item[\rm (8)] Since $(\sqrt{w_1}+\sqrt{\lambda})^2=\tilde{\alpha}_{1}+2\sqrt{\lambda w_1}$, a direct computation shows that
\begin{align*}
g((\sqrt{\omega_1}+\sqrt{\lambda})^2)= -\frac{3}{2\sqrt{\lambda w_1}}+\frac{1}{2\lambda \omega_1\omega}L,
\end{align*}
where
$$
L:=-2\omega(\omega_1-\omega_2)+\omega_2(\omega_1-\omega) +2\omega_2\sqrt{\lambda \omega_1}- \omega_2 \sqrt{(2\sqrt{\lambda \omega_1}+\omega_1-\omega)^2 - 4\lambda \omega}.\\
$$
By Lemma \ref{lem:omega}, we have
\begin{align*}
    L
    &=-2\omega(\omega_1-\omega_2)+\omega_2(\omega_1-\omega)+2\omega_2\sqrt{\lambda \omega_1} \\
    &\hspace{5mm}-\omega_2 \sqrt{(\omega_1-\omega)^2+4\sqrt{\lambda\omega_1}(\omega_1-\omega)+4\lambda(\omega_1-\omega)}\\
    &<-2\omega(\omega_1-\omega_2)+\omega_2(\omega_1-\omega)+2\omega_2\sqrt{\lambda \omega_1} - \omega_2(\omega_1-\omega)\\
    &=-2\omega(\omega_1-\omega_2)+2\omega_2\sqrt{\lambda \omega_1}.
\end{align*}
Therefore we have
\begin{align*}
g((\sqrt{\omega_1}+\sqrt{\lambda})^2) 
&<-\frac{3}{2\sqrt{\lambda w_1}}-\frac{2\omega(\omega_1-\omega_2)}{2\lambda \omega_1\omega}+\frac{2\omega_2\sqrt{\lambda \omega_1}}{2\lambda\omega_1\omega}\\
&=-\frac{2\omega(\omega_1-\omega_2)}{2\lambda \omega_1\omega} + \left(-\frac{3}{2}+\frac{\omega_2}{\omega}\right) \frac{1}{\sqrt{\lambda \omega_1}}<0.
\end{align*}
Finally, we obtain $g((\sqrt{\omega_1}+\sqrt{\lambda})^2)<0$.
\end{enumerate}
\end{proof}

The following proposition is useful for determining the discrete part of the reversed QW-Poisson distribution $\mu_{C_\lambda}$.

\begin{prop}\label{prop:equation_zero}
Let us consider $0<r<1$ and $\lambda>0$. 
\begin{enumerate}[\rm (1)]
\item $0<\lambda \le \Lambda_-(r)$ if and only if the equation $f(x)=0$ has a unique solution $x_0 \in (0 ,\widetilde{\alpha}-2\sqrt{\widetilde{\omega}}]$.

\item $\lambda <\Lambda_+(r)$ if and only if the equation $g(x)=0$ has a unique solution $x_1 \in (\widetilde{\alpha}_1, (\sqrt{\omega_1}+\sqrt{\lambda})^2)$.

\item $\lambda \ge \Lambda_+(r)$ if and only if the equation $g(x)=0$ has a unique solution $x_2 \in [\widetilde{\alpha}+2\sqrt{\widetilde{\omega}},(\sqrt{\omega_1}+\sqrt{\lambda})^2 )$.
\end{enumerate}
\end{prop}

\begin{proof}
\begin{enumerate}[\rm (1)]
\item By Lemma \ref{lem:basicproperties_f_g} (1), we have $\lambda \neq \omega$, and hence $\widetilde{\alpha}-2\sqrt{\widetilde{\omega}}>0$. By Lemma \ref{lem:basicproperties_f_g} (4), (7), and the intermediate value theorem, the condition $f(\widetilde{\alpha}-2\sqrt{\widetilde{\omega}})<0$ is equivalent to the existence of a unique solution $x_0 \in (0,\widetilde{\alpha}-2\sqrt{\widetilde{\omega}})$ to the equation $f(x)=0$. Set $\beta=\sqrt{\lambda}$.
A direct computation shows that
\begin{align*}
        f&\left(\widetilde{\alpha}-2\sqrt{\widetilde{\omega}}\right)\\
        &=\frac{1}{\frac{1-t^2}{4}-(1-t)-r\beta}-\frac{1}{\beta^2(1-t)}\left\{ \frac{1-t^2}{4}-\frac{t(1-t)}{2}-r\beta+ \frac{t(1-t)}{r}\beta \right\} \\
        &=\frac{1}{-\frac{(1-t)(3-t)}{4}-r\beta}-\frac{1}{\beta^2}\left( \frac{1+t}{4}-\frac{t}{2}-\frac{r\beta}{1-t}+ \frac{t}{r}\beta \right) \\
        &=\frac{4}{-(1-t)(3-t)-4r\beta}-\frac{1}{\beta^2}\left( \frac{1+t}{4}-\frac{t}{2}-\frac{r\beta}{1-t}+ \frac{t}{r}\beta \right)\\
        &=\frac{4}{-(1-t)(3-t)-4r\beta}-\frac{1}{\beta^2}\left( \frac{1-t}{4}-\frac{1}{r}\beta \right) 
    \end{align*}
Since $(1-t)(3-t)+4r\beta>0$, the following equivalent inequality holds:
\begin{align*}
     f&\left(\widetilde{\alpha}-2\sqrt{\widetilde{\omega}}\right)\le 0\\
     &\iff  -\frac{1-t}{4}+\frac{\beta}{r}\le\frac{4\beta^2}{(1-t)(3-t)+4r\beta}\\
    &\iff -(1-t)^2(3-t)-4r(1-t)\beta
+\frac{4(1-t)(3-t)}{r}\beta+16\beta^2\le16\beta^2\\
 &\iff -4r(1-t)\beta
+\frac{4(1-t)(3-t)}{r}\beta\le(1-t)^2(3-t)\\
&\iff \frac{4}{r}\biggl(3-t-(1-t^2)\biggr)\beta\le(1-t)(3-t)\\
&\iff \frac{4}{r}(t^2-t+2)\beta\le (1-t)(3-t)\\
&\iff \beta\le\sqrt{\Lambda_-(r)}.
\end{align*}
\end{enumerate}

Next, we show that $g\left(\widetilde{\alpha}+ 2\sqrt{\widetilde{\omega}}\right)>0$ if and only if $\lambda \ge \Lambda_+(r)$, where the value of \(g\) at the endpoint is defined by the left limit
$$
g\left(\widetilde{\alpha}+2\sqrt{\widetilde{\omega}}\right)
:=
\lim_{x\uparrow \widetilde{\alpha}+2\sqrt{\widetilde{\omega}}} g(x)
\in[-\infty,\infty].
$$
Indeed, we obtain
\begin{align}\label{eq:g}
g\left(\widetilde{\alpha}+ 2\sqrt{\widetilde{\omega}}\right) = \frac{4}{-(1-t)(3-t)+4r\beta} - \frac{1}{\beta^2}\left(\frac{1-t}{4} + \frac{\beta}{r}\right).
\end{align}
Set $h(\beta):=-(1-t)(3-t)+4r\beta$. Then
    \begin{align*}
    h(\beta)>0 \quad \iff \quad  \beta>\frac{(1-t)(3-t)}{4r}=\sqrt{\Lambda_+(r)}
    \end{align*}
If $\beta > \sqrt{\Lambda_+(r)}$, then
    \begin{align*}
        g&\left(\widetilde{\alpha}+ 2\sqrt{\widetilde{\omega}}\right)>0\\
        &\iff \frac{1}{\beta^2}\left( \frac{1-t}{4} +\frac{1}{r}\beta\right)<\frac{4}{-(1-t)(3-t)+4r\beta}\\
        &\iff -(1-t)^2(3-t)+4r(1-t)\beta-\frac{4}{r}(1-t)(3-t)\beta+16\beta^2<16\beta^2\\
        &\iff -\frac{4}{r}(t^2-t+2)\beta<(1-t)(3-t).
    \end{align*}
The last inequality always holds, and hence $g\left(\widetilde{\alpha}+ 2\sqrt{\widetilde{\omega}}\right)>0$.
If $\beta = \sqrt{\Lambda_+(r)}$ (i.e. $\lambda = \Lambda_+(r)$), then $g\left(\widetilde{\alpha}+ 2\sqrt{\widetilde{\omega}}\right)= +\infty>0$ since $h(\beta)=0$. 
If $\beta<\sqrt{\Lambda_+(r)}$, then we have $g\left(\widetilde{\alpha}+ 2\sqrt{\widetilde{\omega}}\right)<0$ since the second term in \eqref{eq:g} is negative.
\begin{enumerate}
\item[\rm (2)] Note that $\widetilde{\alpha}+ 2\sqrt{\widetilde{\omega}}< \widetilde{\alpha}_1$ if $\lambda < \Lambda_+(r)$, by Lemma \ref{lem:basicproperties_f_g} (2). By the above discussion, in this case, $g\left(\widetilde{\alpha}+ 2\sqrt{\widetilde{\omega}}\right)<0$. Moreover, $\lim_{x\uparrow \widetilde{\alpha}_1}g(x) =-\infty$. Thus, the equation $g(x)=0$ has no solution $x \in (\widetilde{\alpha}+ 2\sqrt{\widetilde{\omega}}, \widetilde{\alpha}_1)$.

Since $\lim_{x\downarrow \widetilde{\alpha}_1}g(x)=\infty$ and $g((\sqrt{\omega_1}+\sqrt{\lambda})^2))<0$ (by Lemma \ref{lem:basicproperties_f_g} (8)), it follows from Lemma \ref{lem:basicproperties_f_g} (6) and the intermediate value theorem that the equation $g(x)=0$ has a unique solution $x_1\in (\widetilde{\alpha}_1,(\sqrt{\omega_1}+\sqrt{\lambda})^2)$.

\item[\rm (3)] If $\lambda \ge \Lambda_+(r)$, then, by Lemma \ref{lem:basicproperties_f_g} (3), (5), (8) and the above discussion, the equation $g(x)=0$ has a unique solution $x_2\in (\widetilde{\alpha}+2\sqrt{\widetilde{\omega}}, (\sqrt{\omega_1}+\sqrt{\lambda})^2)$.
\end{enumerate}
\end{proof}

\begin{rem}\label{rem:S_-=x_0}
According to the proof in Proposition \ref{prop:equation_zero}, note that, for any $0<r<1$, if $\lambda= \Lambda_-(r)$, then $\widetilde{\alpha}-2\sqrt{\widetilde{\omega}} =x_0$.
\end{rem}

Define the following functions:
    \begin{align*}
        R_{\rm fM_{s,a,b}}(x):=\lim_{\varepsilon\downarrow 0}\Re [G_{{\rm fM}_{s,a,b}}(x+\rmi\varepsilon)]
    \end{align*}
and
\begin{align*}
P_{\rm fM_{s,a,b}}(x):=-\frac{1}{\pi}\lim_{\varepsilon\downarrow 0}\Im [G_{{\rm fM}_{s,a,b}}(x+\rmi\varepsilon)].
\end{align*}
Due to Example \ref{ex:distributions} (2), we obtain
    \begin{align*}
        R_{\rm fM_{s,a,b}}(x)=\frac{(s+2b)x+sa}{2(bx^2+sax+s^2)},\quad x\in \left[a-2\sqrt{s+b},a+2\sqrt{s+b}\right]
    \end{align*}
and
\begin{align}\label{eq:FM}
P_{\rm fM_{s,a,b}}(x)=\frac{s\sqrt{4(s+b)-(x-a)^2}}{2\pi(bx^2+sax+s^2)}\mathbf{1}_{[a-2\sqrt{s+b},a+2\sqrt{s+b}]}(x)
\end{align}

We are now ready to state the one of main results.
\begin{thm}\label{thm:C}
Let $0<r<1$ and $\lambda>0$. The density function of $C_\lambda^{QW}$ is given by
$$
p_{C_\lambda^{QW}}(x)=\frac{\widetilde{\omega}_1 P_{{\rm fM}_{\widetilde{\omega}_2,\widetilde{\alpha}-\widetilde{\alpha}_2, \widetilde{\omega}-\widetilde{\omega}_2}}(x-\widetilde{\alpha}_2)}{\left[ x-\widetilde{\alpha}_1-\widetilde{\omega}_1 R_{{\rm fM}_{\widetilde{\omega}_2,\widetilde{\alpha}-\widetilde{\alpha}_2, \widetilde{\omega}-\widetilde{\omega}_2}}(x-\widetilde{\alpha}_2)\right]^2+ \left[ \pi  \widetilde{\omega}_1 P_{{\rm fM}_{\widetilde{\omega}_2,\widetilde{\alpha}-\widetilde{\alpha}_2, \widetilde{\omega}-\widetilde{\omega}_2}}(x-\widetilde{\alpha}_2)\right]^2}
$$
for all $x\in \calS=[S_-,S_+]:=[\widetilde{\alpha}-2\sqrt{\widetilde{\omega}}, \widetilde{\alpha}+2\sqrt{\widetilde{\omega}}]$.
Let $x_0,x_1,x_2$ be the positive numbers defined in Proposition \ref{prop:equation_zero}. 
\begin{enumerate}[\rm (1)]
\item If $0<\lambda \le \Lambda_-(r)$, then $x_0\in \sigma(C_\lambda^{QW})\cap(0,S_-]$, $x_1\in \sigma(C_\lambda^{QW})\cap (\widetilde{\alpha}_1, (\sqrt{\omega_1}+\sqrt{\lambda})^2)$ and
$$
\mu_{C_\lambda^{QW}}(\rmd x) = \mu_{C_\lambda^{QW}}(\{x_0\})\delta_{x_0}(\rmd x)+\mu_{C_\lambda^{QW}}(\{x_1\})\delta_{x_1}(\rmd x)+p_{C_\lambda^{QW}}(x){\rmd x}.
$$
\item If $\Lambda_-(r) <  \lambda <\Lambda_+(r)$, then $x_1\in \sigma(C_\lambda^{QW})\cap(\widetilde{\alpha}_1,(\sqrt{\omega_1}+\sqrt{\lambda})^2)$ and
$$
\mu_{C_\lambda^{QW}}(\rmd x) = \mu_{C_\lambda^{QW}}(\{x_1\})\delta_{x_1}(\rmd x)+p_{C_\lambda^{QW}}(x){\rmd x}.
$$
\item If $\lambda \ge \Lambda_+(r)$, then $x_2\in \sigma(C_\lambda^{QW}) \cap [S_+, (\sqrt{\omega_1}+\sqrt{\lambda})^2)$ and
$$
\mu_{C_\lambda^{QW}}(\rmd x) = \mu_{C_\lambda^{QW}}(\{x_2\})\delta_{x_2}(\rmd x)+p_{C_\lambda^{QW}}(x){\rmd x}.
$$
\end{enumerate}
\end{thm}
\begin{proof}
The density function is obtained from the Stieltjes-inversion formula (Proposition \ref{prop:Stieltjes}) and \eqref{eq:Cauchy_Poisson}. 
Therefore, the discrete part of $\mu_{C_\lambda^{QW}}$ is supported on $\sigma(C_\lambda^{QW}) \setminus \calS \subset  [0, S_-) \cup (S_+, (\sqrt{\omega_1}+\sqrt{\lambda})^2]$.
By \eqref{eq:Cauchy_Poisson1}, if $\mu_{C_\lambda^{QW}}$ has an atom at $x \in \sigma(C_{\lambda}^{QW})\setminus \calS$, then it suffices to solve the following equation:
\begin{align}\label{eq:atom_equation1}
x-\widetilde{\alpha}_1-\cfrac{\widetilde{\omega}_1}{x-\widetilde{\alpha}_2- \widetilde{\omega_2}G_{{\rm SC}_{\widetilde{\alpha},\widetilde{\omega}}}(x)} =0.
\end{align}
Since the equation \eqref{eq:atom_equation1} is not satisfied at  $x=\widetilde{\alpha}_1$, it is equivalent to 
\begin{align}\label{eq:atom_equation2}
\frac{1}{x-\widetilde{\alpha}_1} - \frac{1}{\widetilde{\omega_1}}\left( x-\widetilde{\alpha}_2- \widetilde{\omega_2}G_{{\rm SC}_{\widetilde{\alpha},\widetilde{\omega}}}(x)\right)=0
\end{align}
for $x \in\sigma(C_\lambda^{QW})\setminus (\calS \cup \{\widetilde{\alpha}_1\})$. Observe that
\begin{align*}
G_{{\rm SC}_{\widetilde{\alpha},\widetilde{\omega}}}(x) 
= \begin{cases}
\dfrac{x-\widetilde{\alpha} + \sqrt{(x-\widetilde{\alpha})^2 -4\widetilde{\omega}}}{2\widetilde{\omega}}, &0<x<S_-\\ \\
\dfrac{x-\widetilde{\alpha} - \sqrt{(x-\widetilde{\alpha})^2 -4\widetilde{\omega}}}{2\widetilde{\omega}}, &x>S_+.
\end{cases}
\end{align*}
Hence the equation \eqref{eq:atom_equation2} is equivalent to $f(x)=0$ on $\sigma(C_\lambda^{QW}) \cap [0,S_-]$, and to $g(x)=0$ on $\sigma(C_\lambda^{QW})\cap [S_+, (\sqrt{\omega_1}+\sqrt{\lambda})^2] \setminus\{\widetilde{\alpha}_1\}$.

\begin{enumerate}[\rm (1)]
\item If $0<\lambda\le \Lambda_-(r)$, then the equation $f(x)=0$ has a unique solution $x_0 \in (0,S_-]$, and also the equation $g(x)=0$ has a unique solution $x_1\in (\widetilde{\alpha}_1, (\sqrt{\omega}_1+\sqrt{\lambda})^2)$. In this case, since $\limsup_{y\downarrow 0}\left|G_{\mu_{C_\lambda}^{QW}}(x_k + \rmi y)\right| = \infty$, we have $x_k \in \sigma(C_\lambda^{QW})$ for $k=0,1$. Moreover, one can see that
$$
\lim_{y\downarrow 0} \rmi y G_{\mu_{C_\lambda^{QW}}}(x_k +\rmi y) >0, \quad k=0,1.
$$
Therefore $\mu_{C_\lambda^{QW}}(\{x_k\})>0$ by Proposition \ref{prop:Stieltjes} (2).
\end{enumerate}
Statements (2) and (3) are obtained by a similar argument.
\end{proof}

\begin{prop}\label{prop:final_spec}
We have $\sigma(C_\lambda^{QW})=\text{supp}(\mu_{C_\lambda^{QW}})$. More precisely, 
\begin{enumerate}[\rm (1)]
\item If $0<\lambda < \Lambda_-(r)$, then $\sigma(C_\lambda^{QW})=\{x_0\}\cup \{x_1\} \cup [S_-, S_+]$.
\item If $\Lambda_-(r) \le \lambda <\Lambda_+(r)$, then $\sigma(C_\lambda^{QW})=\{x_1\} \cup [S_-, S_+]$.
\item If $\lambda \ge \Lambda_+(r)$, then $\sigma(C_\lambda^{QW})=\{x_2\} \cup [S_-, S_+]$. 
\end{enumerate}
where the positive numbers $x_0,x_1,x_2$ were defined in Proposition \ref{prop:equation_zero}. Moreover, $0 \notin \sigma(C_\lambda^{QW})$ (i.e. $S_->0$) if and only if $\lambda \neq \omega$.
\end{prop}
\begin{proof}
Note that the vacuum vector $\Phi_0$ is cyclic for $C_\lambda^{QW}$; that is,
\begin{align}\label{eq:cyclic}
\text{cl}\{p(C_\lambda^{QW}) \Phi_0: p\in \C[x]\}=\calH.
\end{align}
Indeed, by \eqref{eq:C_zenkashiki}, we have
$$
C_\lambda^{QW}\Phi_0 = \sqrt{\lambda \omega_1} \Phi_1 + (\lambda+\omega_1)\Phi_0,
$$
and hence
$$
\Phi_1 = \frac{1}{\sqrt{\lambda \omega_1}}(C_\lambda^{QW}- (\lambda+\omega_1)I)\Phi_0 \in \calD:=\{p(C_\lambda^{QW}) \Phi_0: p\in \C[x]\}.
$$
By induction, using \eqref{eq:C_zenkashiki}, we have $\Phi_n\in \calD$ for all $n\in \N$. Thus, \eqref{eq:cyclic} holds. 

We define an operator $V:\calD \to \C[x]$ by
$$
V\left(p(C_\lambda^{QW})\Phi_0\right) := p(x).
$$
For any $p\in \C[x]$, it follows that
$$
\|p(C_\lambda^{QW})\Phi_0\|^2= \int_{[0,\infty)} |p(x)|^2\mu_{C_\lambda^{QW}}(\rmd x) =\|p\|_{L^2}^2 = \|V(p(C_\lambda^{QW})\Phi_0)\|_{L^2}^2,
$$
where $\|f\|_{L^2}$ is the $L^2$-norm with respect to the measure $\mu_{C_\lambda}^{QW}$. Thus, $V$ is an isometry. Since $\Phi_0$ is cyclic for $C_\lambda^{QW}$ and $\C[x]$ is dense in $L^2(\mu_{C_\lambda^{QW}})$ by Stone--Weierstrass theorem, the operator $V$ extends to a unitary operator $V:\calH\to L^2(\mu_{C_\lambda^{QW}})$. Moreover, $VC_\lambda^{QW} V^{-1}=M_x$ on $\C[x]$, where $M_x$ is the multiplication operator by $x$. Since $\C[x]$ is dense in $L^2(\mu_{C_\lambda^{QW}})$, we have $VC_\lambda^{QW}V^{-1}=M_x$ on $L^2(\mu_{C_\lambda^{QW}})$.
Therefore, $\sigma(C_\lambda^{QW})=\sigma(M_x)=\text{supp}(\mu_{C_\lambda^{QW}})$.
By Theorem \ref{thm:C} and Remark \ref{rem:S_-=x_0}, we obtain the desired result. 
\end{proof}

%SECTION 4

\section{Spectral analysis on QW-Poisson operator}
\label{sec4}

In this section, we investigate the spectral distribution of the Poisson operator on the interacting Fock space associated with a discrete-time quantum walk. For $\lambda>0$, we define the following operator on the interacting Fock space $\calH_{QW}$:
$$
P_\lambda^{QW}:=(B_+^{QW}+\sqrt{\lambda}I)(B_-^{QW}+\sqrt{\lambda}I).
$$
We call the operator $P_\lambda^{QW}$ the {\it QW-Poisson operator} associated with a discrete-time quantum walk. Clearly, $P_\lambda^{QW}$ is a positive operator in $\calA_{QW}$. Thus, there is a unique spectral distribution $\mu_{P_\lambda^{QW}}$ of $P_\lambda^{QW}$ with respect to the vacuum state $\varphi_{QW}$. We call the measure $\mu_{P_\lambda^{QW}}$ the \textit{QW-Poisson distribution}. Since
\begin{align*}
P_\lambda^{QW} \Phi_n = \sqrt{\lambda \omega_{n+1}}\Phi_{n+1}+ (\lambda + \omega_n)\Phi_n + \sqrt{\lambda \omega_n}\Phi_{n-1}, \quad n\ge 0,
\end{align*}
where $\omega_0=0$ and $\omega_n=\omega$ for $n\ge 3$, the Jacobi sequence of $\mu_{P_\lambda^{QW}}$ is given by
\begin{align}\label{eq:Jacobi-QWPOISSON}
J(\mu_{P_\lambda^{QW}}) = 
\begin{pmatrix}
\lambda & \widetilde{\alpha}_1& \widetilde{\alpha}_2&  \widetilde{\alpha}& \widetilde{\alpha}& \cdots\\
\widetilde{\omega}_1 &\widetilde{\omega}_2 & \widetilde{\omega} & \widetilde{\omega} & \widetilde{\omega} & \cdots
\end{pmatrix}.
\end{align}
Thus, analyzing $\mu_{P_\lambda^{QW}}$ is more difficult than analyzing $\mu_{C_\lambda^{QW}}$ because there is a one-step discrepancy between $\widetilde{\alpha}_n$ and $\widetilde{\omega}_n$.

To study the spectral distribution $\mu_{P_\lambda^{QW}}$ without computing its Cauchy transform, we investigate the relation between the Poisson operator and the reversed Poisson operator via a size-biased transform.

\subsection{Size-biased transform}

Let $\mu$ be a probability measure on $[0,\infty)$ with finite mean $m>0$. We define a probability measure $\mathfrak{sb}(\mu)$ by
$$
\mathfrak{sb}(\mu)(\rmd x) := \frac{x\mu(\rmd x)}{m}.
$$
The mapping $\mathfrak{sb}$ between probability measures on $[0,\infty)$, is called the \textit{size-biased transform}. This transform is widely used in probability theory and Stein's method to establish bounds for distributional approximations; see \cite{arratia2019size} and references therein.

This transform is related to Poisson operators on the (general) interacting Fock space. Let $(\calH, \{\Phi_n\}_{n\ge0}, B_+^\tau, B_-^\tau)$ be an interacting Fock space associated with weight $\tau:=\{\tau_n\}_n \subset (0,\infty)$ satisfying $\inf_{n\ge 1} \tau_n>0$ and $\sup_{n\ge 1}\tau_n<\infty$, that is, $\calH$ is a Hilbert space with CONS $\{\Psi_n\}_{n\ge0}$ and 
\begin{align*}
B_+^\tau \Phi_n &:= \sqrt{\tau_{n+1}}\Phi_{n+1}, \qquad n\ge0;\\
B_-^\tau \Phi_n &:= \sqrt{\tau_{n}}\Phi_{n-1}, \quad n\ge1 \qquad \text{and} \qquad B_-^\tau\Phi_0:=0.
\end{align*}
We define the Poisson operator and the reversed Poisson operator associated with the weight $\tau$ as follows:
$$
P_\lambda^\tau:=(B_+^\tau+\sqrt{\lambda}I)(B_-^\tau+\sqrt{\lambda}I)\quad \text{and}\quad C_\lambda^\tau:=(B_-^\tau+\sqrt{\lambda}I)(B_+^\tau+\sqrt{\lambda}I).
$$

For a selfadjoint operator $T$ on $\calH$, we denote by $E_T$ the spectral projection of $T$ and we define the spectral distribution of $T$ by
$$
\mu_T(B):= \langle \Phi_0, E_T(B)\Phi_0\rangle_{\calH}
$$ 
for any Borel sets $B$ in $\R$.

\begin{thm}\label{thm:Poisson_size-biased}
Let $\tau=\{\tau_n\}_{n}$ be a sequence in $(0,\infty)$ such that $\inf_{n\ge 1} \tau_n>0$ and $\sup_{n\ge 1} \tau_n<\infty$, and let $(\calH, \{\Phi_n\}_{n\ge0}, B_+^\tau, B_-^\tau)$ be an interacting Fock space associated with weight $\tau$. Then we have
$$
\mathfrak{sb}\left(\mu_{P_\lambda^{\tau}}\right)=\mu_{C_\lambda^{\tau}}.
$$
\end{thm}
To prove the above theorem, we use a result from operator theory. Define
$$
A_\lambda^\tau:=B_-^\tau + \sqrt{\lambda}I.
$$
Then $P_\lambda^\tau=(A_\lambda^\tau)^\ast A_\lambda^\tau$ and $C_\lambda^\tau = A_\lambda^\tau (A_\lambda^\tau)^\ast$.

\begin{lem}\label{lem:spectral_proj}
For any Borel sets $B$ in $[0,\infty)$ and $\lambda>0$, we have
$$
A_\lambda^{\tau}E_{P_\lambda^{\tau}}(B) =  E_{C_\lambda^{\tau}}(B) A_\lambda^{\tau}.
$$
\end{lem}
\begin{proof}
It is easy to see that 
$$
A_\lambda^{\tau}P_\lambda^{\tau} = A_\lambda^{\tau} \left((A_\lambda^\tau)^\ast A_\lambda^\tau \right) = \left(A_\lambda^\tau (A_\lambda^\tau)^\ast \right) A_\lambda^\tau=  C_\lambda^{\tau} A_\lambda^{\tau}.
$$ 
Therefore we first have $A_\lambda^{\tau}p(P_\lambda^{\tau})=p(C_\lambda^{\tau})A_\lambda^{\tau}$ for every polynomial $p\in \C[x]$. By the continuous functional calculus this extends to every continuous function on $\sigma(P_\lambda^{\tau})\cup\sigma(C_\lambda^{\tau})$, and then, by the Borel functional calculus, to every bounded Borel function $f$ in $[0,\infty)$. In particular, taking $f=\mathbf{1}_B$, we obtain $A_\lambda^{\tau} E_{P_\lambda^{\tau}}(B)=E_{C_\lambda^{\tau}}(B)A_\lambda^{\tau}$
for every Borel set $B$ in $[0,\infty)$.
\end{proof}

\begin{proof}[Proof of Theorem \ref{thm:Poisson_size-biased}]
By Lemma \ref{lem:spectral_proj}, we have $P_\lambda^{\tau} E_{P_\lambda^{\tau}}(B)=(A_\lambda^{\tau})^\ast E_{C_\lambda^{\tau}}(B)A_\lambda^{\tau}$ for any Borel sets $B$ in $[0,\infty)$. Hence
\begin{align*}
\int_B x \mu_{P_\lambda^{\tau}}(\rmd x)
&=\langle\Phi_0, P_\lambda^{\tau} E_{P_\lambda^{\tau}}(B)\Phi_0\rangle_{\calH}\\
&=\langle \Phi_0, (A_\lambda^{\tau})^\ast E_{C_\lambda^{\tau}}(B) A_\lambda^{\tau} \Phi_0\rangle_{\calH}\\
&=\langle A_\lambda^{\tau}\Phi_0,  E_{C_\lambda^{\tau}}(B) A_\lambda^{\tau} \Phi_0\rangle_{\calH}\\
&= \lambda \langle \Phi_0, E_{C_\lambda^{\tau}}(B) \Phi_0\rangle_{\calH} = \lambda \mu_{C_\lambda^{\tau}}(B).
\end{align*}
If we take $B=[0,\infty)$, then the first moment of $P_\lambda^\tau$ is $\lambda$. Therefore, we have
\begin{align*}
\mu_{C_\lambda^\tau}(\rmd x) = \frac{x}{\lambda}\mu_{P_\lambda^\tau}(\rmd x)= \mathfrak{sb}(\mu_{P_\lambda^\tau})(\rmd x).
\end{align*}
\end{proof}

\begin{ex}[Free Fock space]
We consider the case $\tau_n:=1$ for all $n\in \N$, which corresponds to the free Fock space. It is known that $\mu_{P_\lambda^\tau} = {\rm MP}(\lambda)$, which is the \textit{Marchenko-Pastur distribution}:
\begin{align*}
{\rm MP}(\lambda):=& \max\{1-\lambda,0\}\delta_0(\rmd x)+ \frac{\sqrt{(b_\lambda-x)(x-a_\lambda)}}{2\pi x} \mathbf{1}_{[a_\lambda,b_\lambda]}(x)\rmd x,
\end{align*}
where $a_\lambda=(1-\sqrt{\lambda})^2$ and $b_\lambda=(1+\sqrt{\lambda})^2$; see \cite[Theorem 4.24]{obata2017spectral}.
By Theorem \ref{thm:Poisson_size-biased}, we get
\begin{align*}
\mu_{C_\lambda^\tau}(\rmd x)=\frac{\sqrt{(b_\lambda-x)(x-a_\lambda)}}{2\pi \lambda} \mathbf{1}_{[a_\lambda,b_\lambda]}(x)\rmd x.
\end{align*}
\end{ex}

\begin{prop}\label{prop:Cauchy_Poisson}
For any $\lambda>0$, we obtain
\begin{align*}
G_{\mu_{P_\lambda^\tau}}(z)=\frac{1}{z} (1+\lambda G_{\mu_{C_\lambda^\tau}}(z)), \qquad z\in \C^+.
\end{align*}
\end{prop}
\begin{proof}
For a probability measure $\mu$ on $[0,\infty)$ with the first moment $m>0$, we obtain
\begin{align*}
G_{\mathfrak{sb}(\mu)}(z)  
&= \int_{[0,\infty)} \frac{1}{z-x} \frac{x}{m} \mu(\rmd x)\\
&= \frac{1}{m}\int_{[0,\infty)}  \left(-1+\frac{z}{z-x}\right) \mu(\rmd x) = \frac{1}{m}(-1+zG_\mu(z)).
\end{align*}
Applying the above result to the measure $\mu_{P_\lambda^\tau}$ and using Theorem \ref{thm:Poisson_size-biased}, we obtain
$$
G_{\mu_{C_\lambda^\tau}}(z)=\frac{1}{\lambda}\left(-1+ z G_{\mu_{P_\lambda^\tau}}(z)\right),
$$
as desired.
\end{proof}

\subsection{Case of QW-Poisson operator}

In this section, we investigate the spectral distribution of $P_\lambda^{QW}=(B_+^{QW}+\sqrt{\lambda}I)(B_-^{QW}+\sqrt{\lambda}I)$. For simplicity, we define
$$
A_\lambda^{QW}:=B_-^{QW}+\sqrt{\lambda}I.
$$
Then $P_\lambda^{QW} =(A_\lambda^{QW})^\ast A_\lambda^{QW}$ and $C_\lambda^{QW}=A_\lambda^{QW}(A_\lambda^{QW})^\ast$. 

\begin{lem}\label{lem:injective}
\begin{enumerate}[\rm (1)]
\item $A_\lambda^{QW}$ is injective (i.e. $\Ker A_\lambda^{QW}=\{0\}$) if and only if $\lambda \ge \omega$.
\item $0\notin \sigma(P_\lambda^{QW})$ if and only if $\lambda>\omega$.
\end{enumerate}
\end{lem}
\begin{proof}
\begin{enumerate}[\rm (1)]
\item Suppose that $A_\lambda^{QW}\Phi=0$ for $\Phi\in \calH$. Since $\{\Phi_n\}_{n\ge 0}$ is CONS, there exist $\{c_n\}_{n\ge0} \subset \C$ such that $\Phi=\sum_{n\ge0}c_n\Phi_n$. Then
\begin{align*}
A_\lambda^{QW}\Phi
&=\sum_{n\ge0} c_n B_-^{QW}\Phi_n+ \sqrt{\lambda}\Phi\\
&=\sum_{n\ge1} c_n \sqrt{\omega_n}\Phi_{n-1} + \sqrt{\lambda}\Phi\\
&=(c_1\sqrt{\omega_1}+\sqrt{\lambda}c_0)\Phi_0 + \sum_{n\ge 1} (c_{n+1}\sqrt{\omega_{n+1}} + \sqrt{\lambda}c_n)\Phi_n.
\end{align*}
In order for $A_\lambda^{QW}\Phi=0$ to hold, we must have
$$
c_{n+1}\sqrt{\omega_{n+1}}+ \sqrt{\lambda} c_n = 0, \qquad n\ge0.
$$
It follows by induction from the above equations that
$$
c_n = (-1)^n \frac{\lambda^{n/2}}{\sqrt{\omega_1\omega_2\cdots \omega_n}} c_0, \qquad n\ge 1.
$$
We now assume that $c_0\neq 0$.
If $\lambda\geq\omega$, then
\begin{align*}
\|\Phi\|^2=\sum_{n\ge1} |c_n|^2 =\sum_{n\ge1} \frac{\lambda^n}{\omega_1\omega_2\cdots \omega_n} |c_0|^2= \frac{\lambda |c_0|^2}{\omega_1}+ \frac{\lambda^2 |c_0|^2}{\omega_1\omega_2}\sum_{k=0}^\infty \left(\frac{\lambda}{\omega}\right)^k = \infty.
\end{align*}
This means that $\|\Phi\| = \infty$. This contradicts the assumption that $\Phi\in\calH$. Hence, we have $c_0=0$, and therefore $\Phi=0$. Thus, $A_\lambda^{QW}$ is injective.
On the other hands, if $\lambda < \omega$, then $\|\Phi\|^2=\sum_{n\ge 1}|c_n|^2<\infty$. Therefore, $0 \neq \Phi \in \Ker A_\lambda^{QW}$. 
\item Define the following two operators:
\begin{align*}
S_\lambda:= \sqrt{\omega} S + \sqrt{\lambda}I \quad \text{on} \quad\overline{\text{span}}\{\Phi_n: n\ge1\} \qquad \text{and} \qquad 
S_\lambda\Phi_0:=0
\end{align*}
where $ S\Phi_n:=\Phi_{n-1}$ $(n\ge 1)$, and
\begin{align*}
F\Phi_n = \begin{cases}
(\sqrt{\omega_n}- \sqrt{\omega})\Phi_{n-1}, & n=1,2;\\
0, & n=0, 3,4,\dots.
\end{cases}
\end{align*}
Then $A_\lambda^{QW}=S_\lambda+F$. Clearly, $S_\lambda$ is bounded. If $\lambda>\omega$, then for all $\Phi\in \calH$, we obtain
\begin{align*}
\|S_\lambda \Phi\| = \|\sqrt{\omega} S \Phi - \sqrt{\lambda}\Phi\|\ge \sqrt{\lambda} \|\Phi\|- \sqrt{\omega}\|S\Phi\| \ge (\sqrt{\lambda}-\sqrt{\omega})\|\Phi\|.
\end{align*}
Hence $\Ran S_\lambda$ is closed in $\calH$ and $S_\lambda$ is injective. In particular, since $F$ is finite-rank, $\Ran A_\lambda^{QW}=\Ran(S_\lambda+F)$ is also closed in $\calH$. Therefore, $0 \notin \sigma((A_\lambda^{QW})^\ast A_\lambda^{QW})=\sigma(P_\lambda^{QW})$. 

By the statement (1), if $\lambda < \omega$, then $A_\lambda^{QW}$ is not injective. Then there exists $0\neq \Phi \in \calH$ such that $A_\lambda^{QW}\Phi=0$, and hence $P_\lambda^{QW}\Phi=0$. Thus, $0\in \sigma(P_\lambda^{QW})$. If $\lambda=\omega$, then $S_\lambda=\sqrt{\lambda}(S+I)$. We note that $S$ is unitary equivalent to the unilateral shift on $\ell^2(\mathbb{Z}_{\ge 0})$. Let $P_{\Phi_0}$ be the orthogonal projection onto $\C\Phi_{0}$. Since $S^{\ast}S=I-P_{\Phi_0}$ and $SS^{\ast}=I$, it follows that $(S+I)^{\ast}(S+I)=S+S^{\ast}+2I-P_{\Phi_{0}}$. According to \cite[Appendix B]{matsuzawa2023witten}, $\sigma(S+S^{\ast}+2I)=\sigma_{\mathrm{ess}}(S+S^{\ast}+2I)=[0, 4]$. Since $P_{\Phi_0}$ is rank-one, we have $0\in \sigma_{\mathrm{ess}}((S+I)^\ast(S+I))$. Because $P_\lambda^{QW}$ equals to $S_{\lambda}^\ast S_\lambda$ plus compact operators, we get $0\in \sigma_{\mathrm{ess}}(P_\lambda^{QW})\subset \sigma(P_\lambda^{QW})$. Thus, we get the desired results.
\end{enumerate}
\end{proof}
\begin{prop}\label{prop:spectrum_P_C}
$\sigma(P_\lambda^{QW})=\sigma(C_\lambda^{QW})$ if and only if $\lambda\ge \omega$.
\end{prop}
\begin{proof}
By Deift's theorem (see e.g. \cite[Remark 7.12]{arai2018analysis}), for any $\lambda>0$, we get
\begin{align*}
\sigma(P_\lambda^{QW}) \setminus\{0\}= \sigma(C_\lambda^{QW}) \setminus\{0\}.
\end{align*}
If $\lambda>\omega$, then $0\notin\sigma(C_\lambda^{QW})$ by Proposition \ref{prop:final_spec}, and $0\notin\sigma(P_\lambda^{QW})$ by Lemma \ref{lem:injective}. Therefore, we obtain
$$
\sigma(P_\lambda^{QW})=\sigma(P_\lambda^{QW}) \setminus\{0\}= \sigma(C_\lambda^{QW}) \setminus\{0\}=\sigma(C_\lambda^{QW}).
$$
If $\lambda = \omega$, then $0\in\sigma(C_\lambda^{QW})$ by Proposition \ref{prop:final_spec}, and $0\in\sigma(P_\lambda^{QW})$ by Lemma \ref{lem:injective}. Thus, we get
$$
\sigma(P_\lambda^{QW})=\{0\}\cup\sigma(P_\lambda^{QW}) \setminus\{0\}= \{0\}\cup \sigma(C_\lambda^{QW}) \setminus\{0\}=\sigma(C_\lambda^{QW})
$$
If $\lambda<\omega$, then $0\notin \sigma(C_\lambda^{QW})$ by Proposition \ref{prop:final_spec}, but $0\in \sigma(P_\lambda^{QW})$ by Lemma \ref{lem:injective}. Hence we obtain $\sigma(C_\lambda^{QW})\neq \sigma(P_\lambda^{QW})$
\end{proof}

Consequently, we obtain distributional properties of the QW-Poisson distribution by Theorems \ref{thm:C} and \ref{thm:Poisson_size-biased}.
\begin{thm}\label{thm:spectral_P}
For $\lambda>0$, the probability measure $\mu_{P_\lambda^{QW}}$ is given by
$$
\mu_{P_\lambda^{QW}}(\rmd x) = \max\left\{\frac{\omega_1\omega_2(1-\lambda/\omega)}{\lambda^2+\omega_2(\lambda+\omega_1)(1-\lambda/\omega)},0 \right\}\delta_0(\rmd x) + \frac{\lambda}{x}\mu_{C_\lambda^{QW}}\big|_{(0,\infty)}(\rmd x)
$$
\end{thm}
\begin{proof}
Applying Theorem \ref{thm:Poisson_size-biased} to $\mu_{P_\lambda^{QW}}$, we have
\begin{align}\label{eq:sz-QW}
\int_B x \mu_{P_\lambda^{QW}}(\rmd x) = \lambda \mu_{C_\lambda^{QW}}(B)
\end{align}
for all Borel sets $B$ in $[0,\infty)$. If $\lambda > \omega$, we get $0\notin\sigma (C_\lambda^{QW})=\sigma(P_\lambda^{QW})$ by Proposition \ref{prop:spectrum_P_C}. The formula \eqref{eq:sz-QW} implies that
$$
\mu_{P_\lambda^{QW}}(\rmd x) = \frac{\lambda}{x} \mu_{C_\lambda^{QW}}(\rmd x),
$$
which is a probability measure on $(0,\infty)$.
If $\lambda \le \omega$, then the formula \eqref{eq:sz-QW} implies that
$$
\mu_{P_\lambda^{QW}}(B) =\int_{B\cap (0,\infty)} \frac{\lambda}{x} \mu_{C_\lambda^{QW}}(\rmd x) + \mu_{P_\lambda^{QW}}(\{0\})\delta_0(B)
$$
for all Borel sets $B$ in $[0,\infty)$. It remains to determine $\mu_{P_\lambda^{QW}}(\{0\})$. By Proposition \ref{prop:Cauchy_Poisson}, we obtain
\begin{align*}
\mu_{P_\lambda^{QW}}(\{0\})
&=\lim_{z\to 0}z G_{\mu_{P_\lambda^{QW}}}(z) \\
&= 1+\lambda \lim_{z\to 0}G_{\mu_{C_\lambda^{QW}}}(z)\\
&= 1 + \lambda \cdot \frac{-(\lambda+\omega_2)+ (\lambda/ \omega)\omega_2}{\lambda^2+\omega_2(\lambda+\omega_1)(1-\lambda/\omega)}\\
&=\frac{\omega_1\omega_2(1-\lambda/\omega)}{\lambda^2+\omega_2(\lambda+\omega_1)(1-\lambda/\omega)}.
\end{align*}
The last expression is nonnegative and vanishes if and only if $\lambda=\omega$.
Thus, the desired result holds.
\end{proof}

Combining Lemma \ref{lem:basicproperties_f_g} (1), Proposition \ref{prop:equation_zero}, and Theorems \ref{thm:C} and \ref{thm:spectral_P}, we obtain Theorem \ref{thm:mainFUW}.

%SECTION5

\section{Edge behavior of QW-Poisson density}
\label{sec5}

In this section, we investigate the edge behavior of the absolutely continuous part of the QW-Poisson distribution $\mu_{P_\lambda^{QW}}$.

By Theorems \ref{thm:C} and \ref{thm:spectral_P}, the density function of $\mu_{P_\lambda^{QW}}$ is given by
\begin{align*}
\frac{\rmd \mu_{P_\lambda^{QW}}}{\rmd x}(x)=\frac{\lambda}{x}\cdot\frac{\widetilde{\omega}_1 P_{{\rm fM}_{\widetilde{\omega}_2,\widetilde{\alpha}-\widetilde{\alpha}_2, \widetilde{\omega}-\widetilde{\omega}_2}}(x-\widetilde{\alpha}_2)}{\left[ x-\widetilde{\alpha}_1-\widetilde{\omega}_1 R_{{\rm fM}_{\widetilde{\omega}_2,\widetilde{\alpha}-\widetilde{\alpha}_2, \widetilde{\omega}-\widetilde{\omega}_2}}(x-\widetilde{\alpha}_2)\right]^2+ \left[ \pi  \widetilde{\omega}_1 P_{{\rm fM}_{\widetilde{\omega}_2,\widetilde{\alpha}-\widetilde{\alpha}_2, \widetilde{\omega}-\widetilde{\omega}_2}}(x-\widetilde{\alpha}_2)\right]^2}
\end{align*}
for $x\in [S_-,S_+]=[\widetilde{\alpha}-2\sqrt{\widetilde{\omega}}, \widetilde{\alpha}+2\sqrt{\widetilde{\omega}}]=[(\sqrt{\lambda}-r/2)^2, (\sqrt{\lambda}+r/2)^2]$. For simplicity, we denote by
\begin{align*}
    F_1(x)&:=\lambda\sqrt{1-r^2}\sqrt{\left(\left(\sqrt{\lambda}+\frac{r}{2}\right)^2-x\right)\left(x-\left(\sqrt{\lambda}-\frac{r}{2}\right)^2\right)},\\
    F_2(x)&:=(x-\widetilde{\alpha}_2)^2+\omega_2(x-\widetilde{\alpha}_2)+\lambda(1-r^2),\\
    G_1(x)&:=\left\{(x-\widetilde{\alpha}_1)\left((x-\widetilde{\alpha}_2)^2+\omega_2(x-\widetilde{\alpha}_2)+\lambda(1-r^2)\right)-\lambda\left((x-\widetilde{\alpha}_2)+\frac{\omega_1\omega_2}{2}\right)\right\}^2,\\
    G_2(x)&:=\lambda^2(1-r^2)\left(\left(\sqrt{\lambda}+\frac{r}{2}\right)^2-x\right)\left(x-\left(\sqrt{\lambda}-\frac{r}{2}\right)^2\right).
\end{align*}
Note that
\begin{align*}
G_1(x)&=\left\{ (x-\tilde{\alpha}_1)F_2(x) -\lambda\left((x-\tilde{\alpha}_2)+\frac{\omega_1\omega_2}{2}\right)\right\}^2,\\
G_2(x)&=F_1(x)^2.
\end{align*}
and $F_1(x),F_2(x),G_1(x),G_2(x)\geq 0$ for all $x\in [S_-,S_+]$. Consequently, we can write
\begin{align}\label{eq:Poisson_density_short}
    \frac{\rmd\mu_{P_\lambda^{QW}}}{\rmd x}(x)=\frac{\lambda F_1(x)F_2(x)}{\pi x(G_1(x)+G_2(x))}, \qquad x\in [S_-,S_+].
\end{align}
Let us define
$$
\widetilde{S}_\pm:=\widetilde{\alpha}_2-\frac{\omega_2\mp\sqrt{\omega_2^2-4\lambda(1-r^2)}}{2}=\lambda+\frac{\omega_2\pm \sqrt{\omega_2^2-4\lambda(1-r^2)}}{2}, 
$$
where $\lambda\leq\dfrac{\omega_2^2}{4(1-r^2)}=\dfrac{\omega_1^2}{16}$. Then $F_2(\widetilde{S}_\pm)=0$. If $\dfrac{\omega_1^2}{16}<\lambda$, then $F_2(x)>0$ for all $x\in \R$. 
Moreover, let us define
$$
\Lambda_{\rm c}(r) :=\left(\dfrac{r(1-\sqrt{1-r^2})}{4}\right)^2 = \frac{\omega\omega_1^2}{4}, \qquad 0<r<1.
$$
It is easy to see that $\Lambda_{\rm c}(r) < \dfrac{\omega_1^2}{16}$ since $0<r<1$.
In the following, we prepare a few technical lemmas. 

\begin{lem}\label{lem:F_2>0}
Assume that $\lambda\leq \dfrac{\omega_1^2}{16}$. If $\lambda\neq \Lambda_{\rm c}(r)$, then $\widetilde{S}_-\leq \widetilde{S}_+< S_-<S_+$. If $\lambda =\Lambda_{\rm c}(r)$, then $\widetilde{S}_+ = S_-$. Consequently, for any $\lambda \neq \Lambda_{\rm c}(r)$ and $x\in [S_-,S_+]$, we have $F_2(x)> 0$.
\end{lem}
\begin{proof}
It suffices to prove $\widetilde{S}_+<S_-$. We consider the following quadratic polynomial for $\lambda$:
\begin{align*}
q_1(\lambda):=\lambda^2-\frac{\omega\omega_1^2}{2}\lambda+\frac{\omega^2\omega_1^4}{16} = \left(\lambda -\frac{\omega\omega_1^2}{4}\right)^2  \ge 0.
\end{align*}
Notice that $q_1(\lambda) = (\lambda -\Lambda_{\rm c}(r))^2$. Moreover, for any $0<r<1$, we define
\begin{align*}
q_2(\lambda):=(1-2r^2)\lambda + \frac{\omega \omega_1^2}{4}.
\end{align*}
Then $q_2(\lambda)>0$ for all $0<r<1$ and $\lambda \le \omega_1^2/16$. Indeed, if $1-2r^2\ge 0$, it is clear to $q_2(\lambda)>0$. If $1-2r^2<0$, then $q_2(\lambda)$ has the minimum at $\lambda=\omega_1^2/16$, and therefore, for any $\lambda \le \omega_1^2/16$, we have
$$
q_2(\lambda)\ge q_2\left(\frac{\omega_1^2}{16}\right)=(1-2r^2)\frac{\omega_1^2}{16}+\omega\frac{\omega_1^2}{4}=\frac{\omega_1^2}{16}(1-r^2)>0.
$$
Then we get the following equivalent properties of $\widetilde{S}_+\le S_-$ as follows:
    \begin{align*}
        &\widetilde{S}_+\leq S_-\\
        &\iff \lambda+\frac{\omega_2+\sqrt{\omega_2^2-4\lambda(1-r^2)}}{2}\leq \lambda-\sqrt{\lambda}r+\omega\nonumber\\
        &\iff \sqrt{\omega_2^2-4\lambda(1-r^2)}+2\sqrt{\lambda}r\leq 2\omega-\omega_2\nonumber\\
        &\iff \omega_2^2-4\lambda(1-r^2)+4\lambda r^2+4\sqrt{\lambda}r\sqrt{\omega_2^2-4\lambda(1-r^2)}\leq 4\omega^2+\omega_2^2-4\omega\omega_2\nonumber\\
        &\iff \sqrt{\lambda}r\sqrt{\omega_2^2-4\lambda(1-r^2)}\leq (1-2r^2)\lambda+\frac{\omega\omega_1^2}{4}\ (=q_2(\lambda))\\
        &\iff \lambda r^2\left(\omega_2^2-4\lambda(1-r^2)\right)\leq (1-2r^2)^2\lambda^2+\frac{\omega\omega_1^2}{2}(1-2r^2)\lambda+\frac{\omega^2\omega_1^4}{16}\nonumber\\
        &\iff 0\leq \left[4(1-r^2)r^2+(1-2r^2)^2 \right]\lambda^2+\left[ \frac{\omega\omega_1^2}{2}(1-2r^2)-r^2\omega_2^2\right]\lambda+\frac{\omega^2\omega_1^4}{16}\\
        &\iff 0\le \lambda^2 - \frac{\omega\omega_1^2}{2}\lambda + \frac{\omega^2\omega_1^4}{16} = q_1(\lambda).
        \end{align*}
The last inequality always holds as above discussion, and therefore $\widetilde{S}_+\leq S_-$. Since $q_1(\lambda)=0 \iff \lambda=\Lambda_{\rm c}(r)$, we finally obtain the desired result.  
\end{proof}

\begin{lem}\label{lem:G_1=0}
\begin{enumerate}[\rm (1)]
\item For any $0<r<1$, the following conditions are equivalent.
\begin{enumerate}[\rm (i)]
\item $G_1(S_-)=0$;
\item $\lambda = \Lambda_-(r)$ or $\lambda=\Lambda_{\rm c}(r)$.
\end{enumerate}
\item $G_1(S_+)>0$.
\end{enumerate}
\end{lem}
\begin{proof}
The following equivalent properties hold:
\begin{align*}
&G_1(S_{\pm})=0\\
&\iff (S_{\pm}-\widetilde{\alpha}_1)F_2(S_{\pm})-\lambda\left(S_{\pm}-\widetilde{\alpha}_2+\frac{\omega_1\omega_2}{2}\right)=0\\
&\iff g_\pm(\lambda):=\left( \pm\sqrt{\lambda}r+\omega-\omega_1 \right)\left( \left(\pm\sqrt{\lambda}r+\omega-\omega_2 \right)^2+\omega_2\left(\pm\sqrt{\lambda}r+\omega-\omega_2 \right)+\lambda(1-r^2)\right)\\
&\hspace{28mm}-\lambda\left(\pm\sqrt{\lambda}r+\omega-\omega_2 +\frac{\omega_1\omega_2}{2}\right)=0.
\end{align*}
Here, we get
\begin{align*}
    \left(\pm\sqrt{\lambda}r+\omega-\omega_2 \right)^2&+\omega_2\left(\pm\sqrt{\lambda}r+\omega-\omega_2 \right)+\lambda(1-r^2)\\
    =&\pm2\sqrt{\lambda}r(\omega-\omega_2)+(\omega-\omega_2)^2\pm\sqrt{\lambda}r\omega_2+\omega_2(\omega-\omega_2)+\lambda\\
    =&\pm\sqrt{\lambda}r(2\omega-\omega_2)+\omega(\omega-\omega_2)+\lambda.
\end{align*}
Therefore, we have
\begin{align*}
g_\pm(\lambda)
    =&\left( \pm\sqrt{\lambda}r+\omega-\omega_1 \right)\left( \pm\sqrt{\lambda}r(2\omega-\omega_2)+\omega(\omega-\omega_2)+\lambda\right)\\
    &-\lambda\left(\pm\sqrt{\lambda}r+\omega-\omega_2 +\frac{\omega_1\omega_2}{2}\right)\\
    =&(2\omega-\omega_2)r^2\lambda\pm r\omega(\omega-\omega_2)\sqrt{\lambda}\pm r(\omega-\omega_1)(2\omega-\omega_2)\sqrt{\lambda}\\
    &+\omega(\omega-\omega_1)(\omega-\omega_2)+(\omega-\omega_1)\lambda-\lambda\left(\omega-\omega_2+\frac{\omega_1\omega_2}{2}\right)\\
    =&\frac{\omega_1}{2}r^2\lambda\pm r\omega\frac{\omega_1^2}{4}\sqrt{\lambda}\pm r(\omega-\omega_1)\frac{\omega_1}{2}\sqrt{\lambda}+\omega(\omega-\omega_1)\frac{\omega_1^2}{4}-\lambda\left(\omega_1-\omega_2+\frac{\omega_1\omega_2}{2}\right)\\
    =&\frac{\omega_1}{2}\left\{ -\left(1-r^2+\omega_1+\omega_2\right)\lambda \pm r\left(\omega\frac{\omega_1}{2}+\omega-\omega_1\right)\sqrt{\lambda}+\omega(\omega-\omega_1)\frac{\omega_1}{2} \right\}\\
   =&-\frac{\omega_1}{2}\left\{ \left(1-r^2+\omega_1+\omega_2\right)\lambda \mp r\left(\omega\frac{\omega_1}{2}+\omega-\omega_1\right)\sqrt{\lambda}+\omega(\omega_1-\omega)\frac{\omega_1}{2} \right\}  \\
   =&-\frac{\omega_1}{2}\left\{ \frac{1}{2}(t^2-t+2)\lambda \pm \frac{r(1-t)}{8}\left(t^2-2t+5\right)\sqrt{\lambda}+\frac{(1-t)^2(1-t^2)(3-t)}{32} \right\}.
\end{align*}
Thus, we can regard $g_\pm(\lambda)=0$ as a quadratic equation for $\sqrt{\lambda}$. Let $D$ be a discriminant of $g_\pm(\lambda)=0$ and $t=\sqrt{1-r^2}$. Then we obtain
\begin{align*}
    64D&=r^2(1-t)^2(t^2-2t+5)^2-4(t^2-t+2)(1-t)^2(1-t^2)(3-t)\\
    &=(1-t^2)(1-t)^2\left\{ (t^2-2t+5)^2-4(t^2-t+2)(3-t) \right\}\\
    &=(1-t^2)(1-t)^2(t^4-2t^2+1)\\
    &=(1-t^2)^3(1-t)^2>0.
\end{align*}
Therefore, we obtain
$$
G_1(S_-)=0 \iff \lambda = \Lambda_-(r) \quad \text{or} \quad \lambda=\Lambda_{\rm c}(r).
$$
On the other hand,
$$
G_1(S_+)=0 \iff \sqrt{\lambda}=\frac{-\frac{r(1-t)}{8}\left(t^2-2t+5\right)\pm\sqrt{D}}{t^2-t+2}<0
$$
This is impossible since $\lambda>0$. Hence $G_1(S_+)>0$.
\end{proof}

Finally, we describe the edge behavior of the density function in \eqref{eq:Poisson_density_short}. 

\begin{thm}\label{thm:edge_behabior_density}
Let us fix $0<r<1$. The following properties hold:
\begin{enumerate}[\rm (1)]
\item For any $\lambda>0$, we get
$\lim_{x\uparrow S_+} \left(\dfrac{\rmd \mu_{P_\lambda^{QW}}}{\rmd x}\right)'(x) = -\infty$.
\item If $\lambda = \Lambda_-(r)$ or $\lambda = \omega$, then we get $\lim_{x\downarrow S_-} \left(\dfrac{\rmd \mu_{P_\lambda^{QW}}}{\rmd x}\right)'(x) = -\infty$. 
\item If $\lambda \neq \Lambda_-(r)$ and $\lambda \neq \omega$, then we get $\lim_{x\downarrow S_-} \left(\dfrac{\rmd \mu_{P_\lambda^{QW}}}{\rmd x}\right)'(x) = \infty$. 
\end{enumerate}
\end{thm}
\begin{proof}
In this proof, we sometimes omit the independent variable $x$ when writing functions. By \eqref{eq:Poisson_density_short}, a direct computation shows that
\begin{align*}
    \left(\frac{F_1F_2}{G_1+G_2}\right)'
    &=\left(\frac{F_1F_2}{G_1+F_1^2}\right)'\\
    &=\frac{(G_1+F_1^2)(F_1'F_2+F_1F_2')-(G_1'+2F_1'F_1)F_1F_2}{(G_1+F_1^2)^2}\\
    &=\frac{G_1(F_1'F_2+F_1F_2')+F_1^3F_2'-G_1'F_1F_2-F_1'F_1^2F_2}{(G_1+F_1^2)^2}\\
    &=\frac{H}{(G_1+F_1^2)^2},
    \end{align*}
where
$$
H:=G_1(F_1'F_2+F_1F_2')+F_1^3F_2'-G_1'F_1F_2-F_1'F_1^2F_2.
$$
Therefore, we obtain
    \begin{align*}
     \left(\frac{\rmd\mu_{P_\lambda^{QW}}}{\rmd x}(x)\right)'
        &=\left(\frac{\lambda}{\pi x}\frac{F_1F_2}{G_1+G_2}\right)'\\
        &=\frac{\lambda}{\pi}\left(\frac{1}{x}\frac{H}{(G_1+F_1^2)^2}-\frac{1}{x^2}\frac{F_1F_2}{G_1+F_1^2}\right)\\
        &=\frac{\lambda}{\pi}\cdot\frac{xH-F_1F_2(G_1+F_1^2)}{x^2(G_1+F_1^2)^2}.
    \end{align*}
Since we can write $F_1(x)=\lambda\sqrt{1-r^2}\sqrt{(S_{+}-x)(x-S_{-})}$, we have
\begin{align*}
F_1'
&=c\left\{-\frac{1}{2}(S_{+}-x)^{-1/2}(x-S_{-})^{1/2}+\frac{1}{2}(S_{+}-x)^{1/2}(x-S_{-})^{-1/2}\right\}\\
&=-\frac{1}{2(S_{+}-x)}F_1+\frac{1}{2(x-S_{-})}F_1\\
&=\frac{S_{+}+S_{-}-2x}{2(S_+-x)(x-S_-)}F_1\\
&=c\frac{\widetilde{\alpha}-x}{F_1},
\end{align*}
where $c:=\lambda\sqrt{1-r^2}$. Thus, we have
\begin{align}\label{eq:F'limit}
\lim_{x\downarrow S_-}F_1'(x)=+\infty \quad \text{and} \quad \lim_{x\uparrow S_+}F_1'(x)=-\infty.
\end{align}
Note that $F_1(S_\pm)=0$ and the limits $\lim_{x\to S_{\pm}}\frac{\rmd^n}{\rmd x^n}G_1(x)$ and $\lim_{x\to S_{\pm}}\frac{\rmd^n}{\rmd x^n}F_2(x)$ exist for $n=0,1$.

\begin{enumerate}[\rm (1)]
\item  Since $G_1(S_+)>0$ by Lemma \ref{lem:G_1=0} (2), it follows from \eqref{eq:F'limit} that
   \begin{align*}
       \lim_{x\uparrow S_+} [xH(x)-F_1(x)F_2(x)(G_1(x)+F_1(x)^2)]
       &=S_{+}\lim_{x\to S_{+}}H(x)\\
        &=S_{+}F_2(S_{+}) G_1(S_{+})\lim_{x\uparrow S_+}F_1'(x) \\
        &= - \infty.
    \end{align*}
 Moreover, $\lim_{x\uparrow S_+} x^2 (G_1(x)+F_1(x)^2)=S_+^2 G_1(S_+)>0$. Conseuquently, we obtain 
$$
\lim_{x\uparrow S_+}\left(\dfrac{\rmd \mu_{P_\lambda^{QW}}}{\rmd x}\right)' =  \frac{\lambda}{\pi}\cdot\lim_{x\uparrow S_+}\frac{xH(x)-F_1(x)F_2(x)(G_1(x)+F_1(x)^2)}{x^2(G_1(x)+F_1(x)^2)^2} =- \infty.
$$

\item Assume that $\lambda=\Lambda_-(r)$. Note that $G_1(S_-)=0$ by Lemma \ref{lem:G_1=0} (1). Moreover, one can see that $\widetilde{\alpha}-S_- = \Lambda_-(r) \omega -S_->0$ and hence
\begin{align*}
F_1(x) &\asymp \sqrt{x-S_-}  \qquad \text{and} \qquad F_1'(x)=c \frac{\widetilde{\alpha}-x}{F_1(x)} \asymp \frac{1}{\sqrt{x-S_-}};\\
G_1(x) & \asymp (x-S_-)^2  \qquad \text{and} \qquad G_1'(x)  \asymp (x-S_-)
\end{align*}
as $x\downarrow S_-$. Hence there exist $C_1,C_2,C_3>0$ such that
\begin{align*}
H(x)&=G_1(x)(F_1'(x)F_2(x)+F_1(x)F_2'(x))+F_1(x)^3F_2'(x)\\
&\hspace{4mm}-G_1'(x)F_1(x)F_2(x)-F_1'(x)F_1(x)^2F_2(x)\\
&\asymp (x-S_-)^2 \left(\frac{C_1}{\sqrt{x-S_-}}+C_2\sqrt{x-S_-} \right) - C_3\sqrt{x-S_-}\\
&\asymp -\sqrt{x-S_-} \qquad \text{as} \quad x\downarrow S_-.
\end{align*}
Finally, there exist $C_1',C_2',C_3'>0$ such that
\begin{align*}
\left(\dfrac{\rmd \mu_{P_\lambda^{QW}}}{\rmd x}\right)' 
&=\frac{\lambda}{\pi}\cdot\frac{xH(x)-F_1(x)F_2(x)(G_1(x)+F_1(x)^2)}{x^2(G_1(x)+F_1(x)^2)^2}\\
&\asymp \frac{-C_1'\sqrt{x-S_-}}{\{C_2'(x-S_-)^2+C_3'(x-S_-)\}^2}\\
&\asymp -(x-S_-)^{-\frac{3}{2}} \to -\infty \qquad \text{as} \quad x\downarrow S_-.
\end{align*}
Next we asssume that $\lambda =\omega$. In this case, we note that $S_-=0$ and $G_1(S_-)>0$. Then we have 
$$
H(x) \asymp -\sqrt{x}, \qquad F_1(x)F_2(x)(G_1(x)+G_2(x))\asymp x^{\frac{3}{2}}
$$
and
$$
x^2(G_1(x) + G_2(x))^2 \asymp x^4
$$
as $x\downarrow 0$, and therefore
\begin{align*}
\left(\dfrac{\rmd \mu_{P_\lambda^{QW}}}{\rmd x}\right)' 
\asymp  -x^{\frac{5}{2}} \to -\infty, \qquad \text{as} \quad x\downarrow S_-=0.
\end{align*}

\item  First we assume that $\lambda \neq \Lambda_-(r), \omega$. Furthermore, we assume that $ \lambda \neq \Lambda_{\rm c}(r)$. By Lemma \ref{lem:G_1=0} (1), we have $G_1(S_-)>0$. Therefore, it follows from \eqref{eq:F'limit} that
   \begin{align*}
       \lim_{x\downarrow S_-} [xH(x)-F_1(x)F_2(x)(G_1(x)+F_1(x)^2)]=\infty.
    \end{align*}
Moreover, $\lim_{x\uparrow S_-} x^2 (G_1(x)+F_1(x)^2)=S_-^2 G_1(S_-)>0$.
Consequently, we obtain
\begin{align*}
\left(\dfrac{\rmd \mu_{P_\lambda^{QW}}}{\rmd x}\right)'\to \infty, \qquad \text{as} \quad x\downarrow S_-.
\end{align*}
If $\lambda =\Lambda_{\rm c}(r)$, then $\widetilde{\alpha}- S_- = \Lambda_{\rm c} (r) \omega -S_- <0$, and therefore
$$
F_1'(x)=c \frac{\widetilde{\alpha}-x}{F_1(x)} \asymp -\frac{1}{\sqrt{x-S_-}} \qquad \text{as} \quad x\downarrow S_-.
$$
Hence $H(x) \asymp \sqrt{x-S_-}$ as $x\downarrow S_-$ and
\begin{align*}
\left(\dfrac{\rmd \mu_{P_\lambda^{QW}}}{\rmd x}\right)' \asymp  (x-S_-)^{-\frac{3}{2}} \to \infty \qquad \text{as} \quad x\downarrow S_-.
\end{align*}
Finally, the desired result holds. 
\end{enumerate}
\end{proof}

\begin{thm}\label{thm:edge_behavior}
We obtain
$$
\frac{\rmd \mu_{P_\lambda^{QW}}}{\rmd x}(x) \sim C_\lambda (x-S_-)^{\kappa(\lambda)} \qquad x\downarrow S_-,
$$
where $C_\lambda>0$ and
$$
\kappa(\lambda):=\begin{cases}
\dfrac{1}{2}, & \lambda\neq \Lambda_-(r) \quad \text{and} \quad \lambda\neq \omega,\\\\
-\dfrac{1}{2}, & \lambda = \Lambda_-(r) \quad \text{or} \quad \lambda = \omega.
\end{cases}
$$
\end{thm}
\begin{proof}
Recall that $F_1(x) \asymp \sqrt{x-S_-}$, $F_2(S_-)>0$ and $G_2(x)=F_1(x)^2\asymp x-S_-$ as $x\downarrow S_-$. Moreover, as $x\downarrow S_-$, we have
\begin{align*}
G_1(x) \asymp \begin{cases}
    (x-S_-)^2, & \lambda=\Lambda_-(r) \quad \text{or} \quad \lambda=\Lambda_{\rm c}(r),\\
    (x-S_-)^0, & \lambda \neq \Lambda_-(r) \quad \text{and} \quad  \lambda\neq \Lambda_{\rm c}(r).
\end{cases} 
\end{align*}
Therefore, if $\lambda=\Lambda_-(r)$, then there exist $C_1,C_2,C_3>0$ such that
\begin{align*}
\frac{\rmd \mu_{P_\lambda^{QW}}}{\rmd x} (x) &= \frac{\lambda F_1(x)F_2(x)}{\pi x(G_1(x)+G_2(x))}\\
&\asymp \frac{C_1\sqrt{x-S_-}}{C_2(x-S_-)^2 +C_3(x-S_-)} \asymp
\frac{1}{\sqrt{x-S_-}}, \qquad x\downarrow S_-.
\end{align*}
If $\lambda = \omega$, then $S_-=0$ and there exist $C_1',C_2'>0$ such that
\begin{align*}
\frac{\rmd \mu_{P_\lambda^{QW}}}{\rmd x} (x) \asymp \frac{C_1'\sqrt{x}}{C_2'x} \asymp
\frac{1}{\sqrt{x}}, \qquad x\downarrow S_-.
\end{align*}
In the following, we assume that $\lambda \neq \Lambda_-(r)$ and $\lambda \neq \omega$. Furthermore, if $\lambda \neq \Lambda_{\rm c}(r)$, then there exist $C_1'',C_2''>0$ such that
\begin{align*}
\frac{\rmd \mu_{P_\lambda^{QW}}}{\rmd x} (x) \asymp \frac{C_1''\sqrt{x-S_-}}{C_2''} \asymp
\sqrt{x-S_-}, \qquad x\downarrow S_-.
\end{align*}
If $\lambda =\Lambda_{\rm c}(r)$, then $F_2(x) \asymp (x-S_-)$ as $x\downarrow S_-$ via Lemma \ref{lem:F_2>0}. Therefore, there exist $C_1''', C_2''', C_3'''>0$ such that
\begin{align*}
\frac{\rmd \mu_{P_\lambda^{QW}}}{\rmd x} (x) 
\asymp \frac{C_1'''\sqrt{x-S_-}\cdot(x-S_-)}{C_2'''(x-S_-)^2 + C_3'''(x-S_-)} \asymp
\sqrt{x-S_-}, \qquad x\downarrow S_-.
\end{align*}

\end{proof}

Consequently, the parameter values $\lambda=\Lambda_-(r)$ and $\lambda=\omega$ mark phase transitions in the edge behavior of the QW-Poisson density. These values also correspond to phase transitions in the spectral properties of the QW-Poisson distribution. More precisely, $\lambda=\Lambda_-(r)$ is the critical value at which the number of atoms of the QW-Poisson distribution changes between three and a different number. Moreover, $\lambda=\omega$ is the critical value determining whether $\sigma(P_\lambda^{QW})=\sigma(C_\lambda^{QW})$ holds (or whether $\sigma(P_\lambda^{QW})$ contains $0$).

%SECTION6

\section{QW-Poisson operator and noncommutative probability}
\label{sec6}

\subsection{Moments of QW-Poisson distributions}
\label{sec6-1}

For a probability measure $\mu$ on $\R$, we define the moment generating function by
$$
M_\mu(z) = \sum_{n\ge 0} m_n(\mu) z^n,
$$
where $m_0(\mu):=1$ and $m_n(\mu)$ is the $n$-th moment of $\mu$. In particular, if $\mu$ has compacted support $[-L,L]$ for some $L>0$, then
$$
G_\mu(z) = \int_{-L}^L \frac{1}{z-x}\mu (\rmd x)=\frac{1}{z} \sum_{n\ge0} \frac{m_n(\mu)}{z^n} = \frac{1}{z} M_\mu\left(\frac{1}{z}\right), \qquad |z|>L.
$$
Hence, we get
\begin{align}\label{eq:MGF}
M_\mu(z) = \frac{1}{z} G_\mu\left(\frac{1}{z}\right)
\end{align}
for sufficiently small $|z|$. Using this formula, we obtain the following result.

\begin{thm}\label{thm:moments}
For any $\lambda>0$, the moments of $\mu_{P_\lambda^{QW}}$ are recursively determined by the following formula:
\begin{align*}
&\sum_{n\ge 0} m_n(\mu_{P_\lambda^{QW}})z^n \\
&= 1 + \lambda z + \lambda \sum_{n\ge1} z^{n+1}\left( \widetilde{\alpha}_1 + \widetilde{\omega}_1 z +\widetilde{\omega}_1\sum_{k \ge 1}^\infty z^{k+1} \left(\widetilde{\alpha}_2+\widetilde{\omega}_2 z + \widetilde{\omega}_2 \sum_{l\ge1}m_l({\rm SC}_{\widetilde{\alpha},\widetilde{\omega}})z^{l+1} \right)^k\right)^n.
\end{align*}
More precisely, we obtain
\begin{itemize}
\item $m_1(\mu_{P_\lambda^{QW}}) = \lambda$.
\item $m_2(\mu_{P_\lambda^{QW}}) = \lambda (\lambda+\omega_1)$
\item $m_3(\mu_{P_\lambda^{QW}}) = \lambda (\lambda^2 + 3\lambda \omega_1+\omega_1^2)$.
\item $m_4(\mu_{P_\lambda^{QW}}) = \lambda  \left(\lambda^3 + 6 \lambda^2\omega_1+\lambda \omega_1(5\omega_1+\omega_2)+\omega_1^3 \right)$.
\end{itemize}
\end{thm}
\begin{proof}
First, we compute the moment generating function of $\mu_{C_\lambda^{QW}}$. By \eqref{eq:Cauchy_Poisson1} and \eqref{eq:MGF}, for sufficiently small $|z|$, we have
\begin{align*}
M_{\mu_{C_\lambda^{QW}}}(z)
&=\frac{1}{z}\cdot \cfrac{1}{z^{-1}-\widetilde{\alpha}_1-\cfrac{\widetilde{\omega}_1}{z^{-1}-\widetilde{\alpha}_2- \widetilde{\omega}_2G_{{\rm SC}_{\widetilde{\alpha},\widetilde{\omega}}}(z^{-1})}}\\
&=\cfrac{1}{1-\widetilde{\alpha}_1z- \widetilde{\omega}_1z^2 \cfrac{1}{1-\widetilde{\alpha}_2z-\widetilde{\omega}_2z^2 M_{{\rm SC}_{\widetilde{\alpha},\widetilde{\omega}}}(z)}}\\
&= \frac{1}{1-\widetilde{\alpha}_1z -\widetilde{\omega}_1z^2 \sum_{k\ge0}z^k (\widetilde{\alpha}_2+\widetilde{\omega}_2z M_{{\rm SC}_{\widetilde{\alpha},\widetilde{\omega}}}(z))^k}\\
&=\sum_{n\ge 0}\left(\widetilde{\alpha}_1z +\widetilde{\omega}_1z^2 \sum_{k\ge0}z^k (\widetilde{\alpha}_2+\widetilde{\omega}_2z M_{{\rm SC}_{\widetilde{\alpha},\widetilde{\omega}}}(z))^k\right)^n.
\end{align*}
We set $m_l := m_l({\rm SC}_{\widetilde{\alpha},\widetilde{\omega}})$ for $l\ge 0$. By Proposition \ref{prop:Cauchy_Poisson},
\begin{align*}
M_{\mu_{P_\lambda^{QW}}}(z) 
&=1 + \lambda z M_{\mu_{C_\lambda^{QW}}}(z)\\
&=1 + \lambda z \sum_{n\ge 0}\left(\widetilde{\alpha}_1z +\widetilde{\omega}_1z^2 \sum_{k\ge0}z^k (\widetilde{\alpha}_2+\widetilde{\omega}_2z M_{{\rm SC}_{\widetilde{\alpha},\widetilde{\omega}}}(z))^k\right)^n\\
&=1+ \lambda z + \lambda \sum_{n\ge 1}z^{n+1}\left(\widetilde{\alpha}_1+ \widetilde{\omega}_1 z \sum_{k\ge 0}z^k\left(\widetilde{\alpha}_2+ \widetilde{\omega}_2 z \sum_{l\ge 0} m_lz^l\right)^k\right)^n\\
&=1+ \lambda z + \lambda \sum_{n\ge 1}z^{n+1} \left(\widetilde{\alpha}_1+ \widetilde{\omega}_1 z + \widetilde{\omega}_1 \sum_{k\ge 1}z^{k+1} \left(\widetilde{\alpha}_2+\widetilde{\omega_2}z + \widetilde{\omega}_2\sum_{l \ge 1}m_lz^{l+1} \right)^k \right)^n
\end{align*}
By comparing the coefficients in the two generating functions, we explicitly obtain the first four moments.
\end{proof}

\begin{rem}
Since the Jacobi parameters of $P_\lambda^{QW}$ become constant from the fourth level onward, the corresponding infinite constant part of the Jacobi matrix is described by Wigner's semicircle distribution ${\rm SC}_{\widetilde{\alpha},\widetilde{\omega}}$. Therefore, this semicircular part does not contribute to the first four moments; its contribution appears only from the fifth moment onward. This fact is also directly reflected in the preceding theorem.
\end{rem}

\subsection{Poisson approximation of the Konno distribution}
\label{sec6-2}

In classical probability theory, it is well known that if $X_\lambda\sim {\rm Po}(\lambda)$, then
$$
\frac{X_\lambda-\E[X_\lambda]}{\sqrt{\text{Var}(X_\lambda)}}=\frac{X_\lambda-\lambda}{\sqrt{\lambda}} \xrightarrow{\rm d} {\rm N}(0,1)
$$
as $\lambda\to\infty$. From the viewpoint of the Boson Fock space, this means that the Gaussianization of the Poisson operator on the Boson Fock space converges in distribution to the Gaussian operator.

In a similar spirit, we approximate the Konno distribution using the Gaussianization of the QW-Poisson operator:
$$
Z_\lambda^{QW}:=\dfrac{P_\lambda^{QW}-\lambda I}{\sqrt{\lambda \omega_1}}, \quad \lambda>0.
$$

\begin{thm}\label{thm:Konno-Poisson}
For $0<r<1$, we obtain 
$D_{\sqrt{\omega_1}}(\mu_{Z_\lambda^{QW}}) \xrightarrow{\rm w} {\rm K}_r$ as $\lambda\to \infty$.
\end{thm}
\begin{proof}
It is easy to see that 
\begin{align*}
\left\| Z_\lambda^{QW}- \frac{N^{QW}}{\sqrt{\omega_1}} \right\|
&= \left\|\frac{B_+^{QW}B_-^{QW}}{\sqrt{\lambda \omega_1}}+\frac{N^{QW}}{\sqrt{\omega}_1}- \frac{N^{QW}}{\sqrt{{\omega}_1}} \right\| \\
&= \frac{1}{\sqrt{\lambda \omega_1}}\|B_+^{QW}B_-^{QW}\| \le \sqrt{\frac{\omega_1}{\lambda}} \to 0
\end{align*}
as $\lambda \to \infty$, where the last inequality holds by Lemma \ref{lem:bounded}. Therefore, all spectra $\sigma(Z_\lambda^{QW})$ and $\sigma(N^{QW})$ are contained in a common compact set $K\subset\mathbb R$ for all sufficiently large $\lambda$. 
For $f\in C(K)$, the operators $f(Z_\lambda^{QW})$ and $f(\omega_1^{-1/2}N^{QW})$ are understood as the continuous functional calculus applied to the restrictions $f|_{\sigma(Z_\lambda^{QW})}$ and $f|_{\sigma(\omega_1^{-1/2}N^{QW})}$, respectively. Thus, the norm convergence $Z_\lambda^{QW}\to \omega_1^{-1/2}N^{QW}$ implies $\|f(Z_\lambda^{QW})- f(\omega_1^{-1/2}N^{QW})\|\to 0$.
Consequently, the spectral distributions of $Z_\lambda^{QW}$ with respect to $\Phi_0$ converge weakly to that of $\omega_1^{-1/2}N^{QW}$, that is, $D_{1/\sqrt{\omega_1}}({\rm K}_r)$.
\end{proof}

\subsection{Boolean self-decomposability}
\label{sec6-3}

Noncommutative probability theory provides a framework for studying random variables that do not necessarily commute. In contrast to classical probability theory, this noncommutativity leads to several different notions of independence. Among them, tensor ($\ast$), free ($\boxplus$), Boolean ($\uplus$), monotone ($\triangleright$) and anti-monotone ($\triangleleft$) independences play particularly important roles. For $\star\in\{\ast,\boxplus,\uplus,\triangleright,\triangleleft\}$, the \textit{$\star$-convolution} of two probability measures is defined as the distribution of the sum of two $\star$-independent random variables having the respective distributions. This distribution depends only on the marginal distributions of the random variables. In particular, it was shown in \cite{Speicher} that, for any probability measures $\mu$ and $\nu$ on $\mathbb{R}$, their Boolean convolution $\mu \uplus \nu$ is characterized by
$$
F_{\mu\uplus\nu}(z) = F_\mu(z)+F_\nu(z)-z, \qquad z\in \C^+,
$$
where $F_\mu$ is the reciprocal Cauchy transform, that is, $F_\mu(z) = 1/G_\mu(z)$.

Once a convolution is defined, it is natural to study fundamental properties of probability distributions with respect to that convolution, such as infinite divisibility and its important subclasses, including self-decomposability and stability. For details in classical probability, see \cite{ken1999levy} and the references therein. For the corresponding theory in free probability, see \cite{bercovici1999stable,BV93,hasebe2019normal,hasebe2016unimodality}.
Actually, infinite divisibility, stability, and self-decomposability have also been defined in the context of Boolean convolution. However, it was proved in \cite[Theorem 3.6]{Speicher} that every probability measure on $\mathbb{R}$ is Boolean infinitely divisible. Therefore, we restrict our attention to the more specific classes of Boolean self-decomposable and Boolean stable distributions. Boolean stable distributions have been studied in, for example, \cite{arizmendi2014classical, arizmendi2016classical}, while Boolean self-decomposable distributions were investigated more recently by the second author in \cite{hasebe2022boolean}. We focus on the class of Boolean self-decomposable distributions as follows.

\begin{defn}
A probability measure $\mu$ on $\R$ is said to be \textit{Boolean self-decomposable} if for any $c\in (0,1)$, there exists a probability measure $\mu_c$ on $\R$, such that
$$
\mu = D_c(\mu)\uplus \mu_c,
$$
where $D_c(\mu)$ denotes the dilation of $\mu$ by the mapping $x\mapsto cx$, that is, $D_c(\mu)(B):= \mu(\{c^{-1}x :x\in B\})$ for any Borel sets $B$ in $\R$.
\end{defn}
For a probability measure $\mu$ on $\R$, we define
$$
\eta_\mu(z) = 1-zF_\mu\left(\frac{1}{z}\right).
$$
In \cite{Speicher}, it was shown that, there exist unique $\gamma\in \R$, $a\ge 0$ and a positive measure $\nu$ on $\R$ satisfying $\nu(\{0\})=0$ and $\int_\R (1\land x^2)\nu(\rmd x)<\infty$, such that
$$
\eta_\mu(z) = \gamma z + az^2 + \int_{\R}\left(\frac{1}{1-zt}-1-zt \mathbf{1}_{[-1,1]}(t)\right)\nu(\rmd t). 
$$
According to \cite[Propositions 3.3 and 3.4]{hasebe2022boolean}, the characterization of Boolean self-decomposable distributions were obtained.

\begin{prop}\label{prop:boolean_k}
\begin{enumerate}[\rm (1)]
\item A probability measure $\mu$ is Boolean self-decomposable if and only if the measure $\nu$ is Lebesgue absolutely continuous and the function
$$
k_\mu:\R\setminus\{0\} \to [0,\infty] \qquad  \text{defined by} \qquad k_\mu(x) = |x|\frac{\rmd \nu^{ac}}{\rmd x}(x),
$$
has a version with respect to the Lebesgue measure that is non-decreasing on $(-\infty,0)$  and non-increasing on $(0,\infty)$. 
\item For a probability measure $\mu$, the function $k_\mu$ is obtaiend by the following formula:
$$
k_\mu(x) = \lim_{\epsilon \downarrow 0} \frac{1}{\pi |x|} \Im[F_\mu(x+\rmi \epsilon)] \qquad \text{a.e. } x\in \R\setminus\{0\}.
$$
\end{enumerate}
\end{prop}

\subsubsection{Konno distribution}

We prove the Boolean self-decomposability for the Konno distribution (i.e. the QW-gaussian distribution). 

For $m\in \R$, we say that a function $k: \R\setminus\{m\}\to \R$ is \textit{unimodal with mode $m$} if $k(x)$ is non-decreasing on $(-\infty, m)$ and non-increasing on $(m,\infty)$. The following lemma is easy, but it is useful to prove the unimodality of functions.
\begin{lem}\label{lem:monotonicity}
If $f:\R\setminus\{0\} \to [0,\infty)$ is differentiable and unimodal with mode $0$, the function $|x|^{-1}f(x)$ is also unimodal with mode $0$.
\end{lem}
\begin{proof}
A simple computation shows the above result. Indeed, for $x>0$, we have
$$
\frac{\rmd}{\rmd x}\left(\frac{f(x)}{x} \right) = \frac{xf'(x)- f(x)}{x^2}<0,
$$
as desired. Similarly, $|x|^{-1}f(x)$ is non-decreasing on $(-\infty,0)$.
\end{proof}

We study the unimodality for the density functions of free Meixner distributions as follows.

\begin{prop}\label{prop:fMUnimodality}
Let us consider $s>0$, $a\in \R$ and $b>0$. We denote by $P_{{\rm fM}_{s,a,b}}(x)$ the density function of free Meixner distribution, defined by \eqref{eq:FM}.
\begin{enumerate}[\rm (1)]
\item If $a=0$, then $P_{{\rm fM}_{s,0,b}}(x)$ is unimodal with mode $0$.
\item If $a\ge 2\sqrt{b}$, then there exists a unique $u \in (a-2\sqrt{s+b}, a+2\sqrt{s+b})$ such that $P_{{\rm fM}_{s,a,b}}(x)$ is unimodal with mode $u$.
\end{enumerate}
\end{prop}
\begin{proof}
\begin{enumerate}[\rm (1)]
\item Since $P_{{\rm fM}_{s,0,b}}(x)$ is symmetric with respect to the origin, it suffices to show that it is non-increasing on $(0,\infty)$. A direct computation shows that
\begin{align*}
\frac{\rmd}{\rmd x} P_{{\rm fM}_{s,0,b}}(x) = \frac{s}{2\pi} \cdot \frac{bx^3-\{s^2+8b(s+b)\}x}{(bx^2+s^2)\sqrt{4(s+b)-x^2}}.
\end{align*}
Since $0<x\le 2 \sqrt{s+b}$, we have
$$
x^3-\{s^2+8b(s+b)\}x \le -8b (s+b)\sqrt{s+b} -2s^2\sqrt{s+b}<0.
$$
Therefore $ P_{{\rm fM}_{s,0,b}}(x)$ is non-increasing on $(0,\infty)$,
\item We define the function
$$
h(x) := \frac{\sqrt{4(s+b)-(x-a)^2}}{bx^2+sax+s^2}, \qquad a-2\sqrt{s+b}\le x \le a+ 2\sqrt{s+b}.
$$
Note that $P_{{\rm fM}_{s,a,b}}(x) = \dfrac{s}{2\pi} h(x)$.
Then we have
$$
h'(x)= \frac{-(x-a)(bx^2+sa x+s^2) - (4(s+b)-(x-a)^2)(2bx +sa)}{(bx^2+sax+s^2)^2\sqrt{4(s+b)-(x-a)^2}}.
$$
We also define $p(x) =-(x-a)(bx^2+sa x+s^2) - (4(s+b)-(x-a)^2)(2bx +sa)$. Clearly, we have
$$
p(a-2\sqrt{s+b})>0 \quad \text{and} \quad p(a+2\sqrt{s+b})<0.
$$
Moreover, since
\begin{align*}
p'(x) &= 3bx^2-6ab x + \{-s^2-sa^2-8b(s+b)+2a^2b \}\\
&=3b(x-a)^2 - \{s^2+sa^2+8b(s+b)+a^2b\}
\end{align*}
and
$$
p'(a-2\sqrt{s+b})=p'(a+2\sqrt{s+b})= -s^2-(a^2-4b)(s+b),
$$
we get
$$
\max_{x\in [a-2\sqrt{s+b}, a+ 2\sqrt{s+b}]}p'(x)= -s^2-(a^2-4b)(s+b).
$$
Due to $a\ge 2\sqrt{b}$, we get $a^2-4b\ge 0$, and therefore, $p'(x)<0$. By the intermediate value theorem, there exists a unique $u\in (a-2\sqrt{s+b},a+2\sqrt{s+b})$ such that $p(u)=0$. Hence $P_{{\rm fM}_{s,a,b}}(x)= \dfrac{s}{2\pi}h(x)$ is unimodal with mode $u$.
\end{enumerate}
\end{proof}

\begin{thm}\label{thm:BSD__Konno}
For any $0<r<1$, the Konno distribution ${\rm K}_r=\mu_{N^{QW}}$ is Boolean self-decomposable.
\end{thm}
\begin{proof}
Notice that,
$$
G_{{\rm K}_r}(z) = \frac{1}{z-\omega_1 G_{{\rm fM}_{\omega_2, 0, \omega-\omega_2}}(z)}.
$$
Since ${\rm K}_r$ is symmetric, the function $k_{{\rm K}_r}$ is also symmetric with respect to the origin. Thus, it suffices to prove that $k_{{\rm K}_r}$ is non-increasing on $(0,\infty)$. By Proposition \ref{prop:boolean_k} (2) and Proposition \ref{prop:Stieltjes} (Stieltjes-inversion formula), for all $x>0$,
\begin{align*}
k_{{\rm K}_r}(x)=\frac{1}{\pi x}\lim_{\epsilon \downarrow 0}\Im [F_{{\rm K}_r}(x+\rmi \epsilon)]=\frac{1}{x} P_{{\rm fM}_{\omega_2,0,\omega-\omega_2}}(x).
\end{align*}
Due to Lemma \ref{lem:monotonicity} and Proposition \ref{prop:fMUnimodality} (1), we obtain that $k_{{\rm K}_r}(x)=x^{-1} P_{{\rm fM}_{\omega_2,0,\omega-\omega_2}}(x)$ is non-increasing on $(0,\infty)$. Finally, ${\rm K}_r$ is Boolean self-decomposable.
\end{proof}

\subsubsection{Reversed QW-Poisson distribution}

We investigate the Boolean self-decomposability for the reversed QW-Poisson distribution $\mu_{C_\lambda^{QW}}$.
\begin{thm}\label{thm:BSD_CPoisson}
For any $0<\lambda\le \dfrac{\omega-\omega_2}{4}$, there exists $u=u(\lambda) \in (\widetilde{\alpha}-\widetilde{\alpha}_2-2\sqrt{\widetilde{\omega}},\widetilde{\alpha}-\widetilde{\alpha}_2+2\sqrt{\widetilde{\omega}})$ such that $\mu_{-(u+\widetilde{\alpha}_2)I + C_\lambda^{QW}}=\delta_{-(u+\widetilde{\alpha}_2)} \ast \mu_{C_\lambda^{QW}}$ is Boolean self-decomposable.
\end{thm}
\begin{proof}
By an argument similar to the proof of Theorem \ref{thm:BSD__Konno}, we have
$$
k_{\mu_{C_\lambda^{QW}}}(x)= \frac{1}{|x|} P_{{\rm fM}_{\widetilde{\omega}_2,\widetilde{\alpha}-\widetilde{\alpha}_2, \widetilde{\omega}-\widetilde{\omega}_2}}(x-\widetilde{\alpha}_2), \quad x\in (\widetilde{\alpha}-2\sqrt{\widetilde{\omega}},\widetilde{\alpha}+2\sqrt{\widetilde{\omega}})\setminus\{0\}.
$$
By Proposition \ref{prop:fMUnimodality} (2), for any $0<\lambda\le \dfrac{\omega-\omega_2}{4}$ (i.e. $\widetilde{\alpha}-\widetilde{\alpha}_2 \ge 2\sqrt{\widetilde{\omega}-\widetilde{\omega}_2}$), there exists a unique $u=u(\lambda)\in (\widetilde{\alpha}-\widetilde{\alpha}_2-2\sqrt{\widetilde{\omega}},\widetilde{\alpha}-\widetilde{\alpha}_2+2\sqrt{\widetilde{\omega}})$ such that the function $P_{{\rm fM}_{\widetilde{\omega}_2,\widetilde{\alpha}-\widetilde{\alpha}_2, \widetilde{\omega}-\widetilde{\omega}_2}}(t)$ is unimodal with mode $u$. Thus, $M(t):= P_{{\rm fM}_{\widetilde{\omega}_2,\widetilde{\alpha}-\widetilde{\alpha}_2, \widetilde{\omega}-\widetilde{\omega}_2}}(t+u)$ is also unimodal with mode $0$. By \cite[Lemma 3.14]{hasebe2022boolean}, we get
\begin{align*}
k_{\delta_{-(u+\widetilde{\alpha}_2)} \ast \mu_{C_\lambda^{QW}}}(x) = \frac{1}{|x|} P_{{\rm fM}_{\widetilde{\omega}_2,\widetilde{\alpha}-\widetilde{\alpha}_2, \widetilde{\omega}-\widetilde{\omega}_2}}(x+u)=\frac{M(x)}{|x|}, \qquad \text{for a.e. } x\in \R.
\end{align*}
By Lemma \ref{lem:monotonicity}, the function $x^{-1} M(x)$ is non-increasing for a.e. $x \in (0,\infty)$. Replacing $x$ with $-x$ and applying Lemma \ref{lem:monotonicity}, the function $|x|^{-1} M(x)$ is non-decreasing for a.e. $x\in (-\infty,0)$. Hence, $\delta_{-(u+\widetilde{\alpha}_2)} \ast \mu_{C_\lambda^{QW}}$ is Boolean self-decomposable.
\end{proof}

\begin{rem}
The measure $\mu_{C_\lambda^{QW}}$ is not Boolean self-decomposable. Indeed, the function $P_{{\rm fM}_{\widetilde{\omega}_2,\widetilde{\alpha}-\widetilde{\alpha}_2, \widetilde{\omega}-\widetilde{\omega}_2}}(x-\widetilde{\alpha}_2)$ is not unimodal with mode $0$, and so is $k_{\mu_{C_\lambda^{QW}}}(x)$.
\end{rem}

%SECTION7

\section{Concluding remarks on the asymmetric case}\label{sec8}

In this paper, we have studied the Poisson operator on the interacting Fock space associated with a discrete-time quantum walk in the case $c(a,b,\varphi)=0$, which corresponds to the symmetric case. It is natural, however, to ask how the Poisson operator should be studied when $c(a,b,\varphi)\neq 0$, namely in the asymmetric case.

In the asymmetric case, apart from some special choices of the parameter $c(a,b,\varphi)$, the Jacobi parameters of the limiting distribution associated with the quantum walk may fail to be eventually constant. In particular, the corresponding Jacobi matrix may no longer contain a semicircular part; see \cite[Section 4]{hamada2009orthogonal}.
For this reason, the method used in the symmetric case does not seem to extend directly. For instance, the argument leading to Lemma \ref{lem:basicproperties_f_g} relies on the eventual constancy of the Jacobi parameters, and an analogous construction may not be available in the asymmetric case.

The behavior of the sequence $\{\omega_n\}_{n\ge1}$ is also expected to play an important role. If $\omega_n$ converges, then results analogous to Lemma \ref{lem:injective} (2) and Proposition \ref{prop:spectrum_P_C} may still hold, at least in a modified form. On the other hand, if $\omega_n$ does not converge, then the corresponding spectral behavior may become more delicate, and the same classification may no longer be valid.

In particular, the distinction at $\lambda=\omega$ in Proposition \ref{prop:spectrum_P_C} is expected to arise when the sequence $\omega_n$ eventually becomes constant with value $\omega$, or more generally when $\omega_n$ converges to $\omega$. If $\omega_n$ has no limit, then one should not expect such a simple phase transition at a single critical value.

Finally, the moment structure may also change substantially in the asymmetric case. In the symmetric case studied in this paper, the Jacobi parameters become constant from a finite level onward, and hence the semicircular part begins to influence the moments from the fifth moment onward, as reflected in Theorem \ref{thm:moments}. In contrast, if the sequence $\omega_n$ does not become eventually constant in the asymmetric case, then there may be no semicircular contribution. In that situation, the higher moments are expected to reflect the structure of the quantum walk more directly.

These observations indicate that the spectral analysis of the QW-Poisson operator in the asymmetric case remains an important problem for future research. In particular, it would be interesting to clarify how the behavior of the Jacobi parameters affects the absolutely continuous part, possible atoms, phase-transition phenomena, and moment structure of the corresponding the QW-Poisson distribution.

\section*{Acknowledgement} 
Daiju Funakawa was supported by JSPS KAKENHI Grant-in-Aid with No.~24K06851.
Yuki Ueda was supported by JSPS KAKENHI Grant-in-Aid for Young Scientists, Grant No.~22K13925. Kazuyuki Wada was supported by JSPS KAKENHI Grant-in-Aid with No.~26K06905.
This work was supported by the Research Institute for Mathematical Sciences,
an International Joint Usage/Research Center located in Kyoto University.

\bibliographystyle{abbrv}
\bibliography{FUW.bib}

\subsection*{Author Information}
\begin{enumerate}
\item[] {\bf Daiju Funakawa}\\
Department of Electronics and Information Engineering,
Hokkai-Gakuen University, Sapporo, Hokkaido 062-8605, Japan.\\
E-mail: funakawa@hgu.jp\\

\item[] {\bf Yuki Ueda}\\
Faculty of Education, Department of Mathematics, Hokkaido University of Education, 9 Hokumon-cho, Asahikawa, Hokkaido 070-8621, Japan.\\
E-mail: ueda.yuki@a.hokkyodai.ac.jp\\

\item[] {\bf Kazuyuki Wada}\\
Faculty of Education, Department of Mathematics, Hokkaido University of Education, 9 Hokumon-cho, Asahikawa, Hokkaido 070-8621, Japan.\\
E-mail: wada.kazuyuki@a.hokkyodai.ac.jp
\end{enumerate}

\end{document}